\documentclass[11pt]{article}

\usepackage{import}

\usepackage{amsmath, amssymb, amscd, amsthm, amsfonts}
\usepackage{graphicx}
\usepackage{hyperref}
\usepackage{mathtools}
\usepackage{tikz}
\usepackage{enumitem}
\usepackage[margin=3cm]{geometry}
\usepackage[english]{cleveref}
\usetikzlibrary{cd}
\usepackage[
backend=biber,
style=alphabetic,
]{biblatex}
\hypersetup{
    colorlinks=true,
    linkcolor=blue,
    citecolor=blue
}

\theoremstyle{plain}
\newtheorem{theorem}{Theorem}[section]
\newtheorem{proposition}[theorem]{Proposition}
\newtheorem{lemma}[theorem]{Lemma}
\newtheorem{corollary}[theorem]{Corollary}

\theoremstyle{definition}
\newtheorem{definition}[theorem]{Definition}
\newtheorem{remark}[theorem]{Remark}
\newtheorem{example}[theorem]{Example}
\numberwithin{equation}{section}

\newcommand{\id}{\mathrm{id}}

\newcommand{\blambda}{\bar\lambda}

\newcommand{\bz}{\bar z}
\newcommand{\bpartial}{\bar{\partial}}

\newcommand{\im}{\mathrm{im}}

\newcommand{\bara}{\bar a}
\newcommand{\barb}{\bar b}

\newcommand{\Hom}{\mathrm{Hom}}

\newcommand{\mca}{\mathcal{A}}

\newcommand{\holI}{\mathcal{I}^\mathrm{hol}}
\newcommand{\mult}{\mathrm{mult}}

\newcommand{\hol}{\mathrm{hol}}

\title{Integration of the elliptic tangent bundle and\\
elliptic Poisson structures}
\author{Bas Wensink\thanks{{\tt b.p.wensink@uu.nl}, Department of Mathematics, Utrecht University, The Netherlands}}

\begin{document}
\maketitle
\begin{abstract}
We explicitly construct a Lie groupoid integrating the elliptic tangent bundle associated to a (possibly normal crossing) elliptic divisor, providing a necessary and sufficient topological condition for the existence of a Hausdorff integration. We also produce an explicit local model for the symplectic integration of an elliptic Poisson structure.
\end{abstract}
\vskip12pt
\noindent
MSC classification 2020: 22A22, 53D17.\\
Subject classification: Lie groupoids.

\tableofcontents
\section{Introduction}
It often occurs in geometry that a (real) manifold $M$ has a distinguished codimension two (real) submanifold $D\,$. The prime example of this is the theory of divisors in complex geometry, where the study of complex codimension one submanifolds is very strongly related to the study of holomorphic line bundles. In this paper we will be concerned with \emph{elliptic divisors} $D\subseteq M$ and their associated \emph{elliptic tangent bundles} $\mathcal{A}_{|D|}\to M\,,$ which will introduce in Chapter \ref{prel}. They find their origins in the study of stable generalised complex structures \cite{CG15,CKW22}.

Ellipic tangent bundles can be thought of as simultaneously a real analogue of holomorphic log tangent bundles of Deligne \cite{Deligne70} and a codimension two analogue of the $b$-tangent bundles of Melrose \cite{Melrose}. They are Lie algebroids whose anchor maps are vector bundle isomorphisms almost everywhere. Such Lie algebroids are always integrable to a (possibly non-Hausdorff) Lie groupoid \cite{Debord,CF03}.

One of the main properties of $b$-tangent bundles is that they naturally give rise to a pseudodifferential calculus, which has applications to analysis on manifolds with boundaries \cite{Melrose}. Later, it was shown by Nistor, Weinstein and Xu that this pseudodifferential calculus could be interpreted as living on a Lie groupoid integrating the $b$-tangent bundle \cite{NWX99}. By analogy, one could try to associate a pseudodifferential calculus to elliptic tangent bundles, which should therefore live on a groupoid that integrates $\mathcal{A}_{|D|}\,$. Such a theory could possibly have applications to analysis of differential operators with certain singular behaviour along a codimension two submanifold.

Moreover, Lie algebroids can be viewed as an infinitesimal counterpart to Lie groupoids, so one philosophy is that the study of geometric structures on $\mathcal{A}_{|D|}$ is closely related to the study of similar geometric structures on an integrating groupoid. Therefore, having an explicit description of a Lie groupoid integrating the elliptic tangent bundle seems desirable.

So far, results in this direction have been obtained by Gualtieri, Li and Pym, where they provide an explicit integration of holomorphic logarithmic tangent bundles coming from complex submanifolds \cite{GLP18}, but the integration of elliptic tangent bundles coming from normal crossing elliptic divisors on real manifolds has so far not appeared in literature.

In this paper, we provide an explicit description of a Lie groupoid integrating the elliptic tangent bundle of a normal crossing elliptic divisor. We will also provide results on the existence of Hausdorff integrations, see Theorems \ref{ExistenceOfHausdorffIntegration} and \ref{ExistenceOfHausdorffIntegrationII}, an explicit description of the final $s$-connected Hausdorff integration of the elliptic tangent bundle in a tubular neighbourhood of a smooth elliptic divisor, see Definition \ref{CoorientableDoubleCover} and Proposition \ref{FinalHausdorff}, and, more generally, an explicit description of the final $s$-connected Hausdorff integration of any normal crossing elliptic divisor in a tubular neighbourhood of any intersection locus of $D\,$, see Definition \ref{HLM} and Proposition \ref{FinalHausdorffII}. We will then equip this groupoid with an elliptic ideal that is compatible with the one on $(M,D)\,,$ see Proposition \ref{ellipticideal}. The integration method shall be based on the integration of $b$-tangent bundles and smooth holomorphic log tangent bundles as described in \cite{GL12} and \cite{GLP18}, respectively.

Finally, we will briefly discuss the symplectic integration of a Poisson structure associated to a symplectic structure on the elliptic tangent bundle, as in \cite{CG15}. The symplectic groupoid of such an \emph{elliptic Poisson structure} will not be the same groupoid as obtained above, as it will be an integration of a certain cotangent Lie algebroid, whereas the elliptic tangent bundle is not a cotangent Lie algebroid!\\

\noindent The organisation is as follows:

In Chapter \ref{prel}, we will briefly introduce the Lie algebroids that are important in this paper.

\noindent Then we turn to studying integrations of elliptic tangent bundles, which we will do in several steps, going from more special cases to the general case of a normal crossing elliptic divisor

In Chapter \ref{Coorientable}, we will study the case where $D\subseteq M$ is a smooth, coorientable, embedded divisor, which will give a Hausdorff Lie groupoid. \emph{Mutatis mutandis}, the construction in this chapter comes from \cite{GL12,GLP18}.

In Chapter \ref{Coorientation double cover}, we will study the first new difficulty in doing the general construction: the case where the divisor $D\subseteq M$ is a smooth, embedded and non-coorientable, but such that there is a normal covering $\pi:M_N\to M\,,$ such that $D_N:=\pi^{-1}(D)$ is coorientable. The construction will provide a Hausdorff Lie groupoid. We will also discuss in detail in the case where $\pi$ is a double cover, which will be a guiding example that will become vital in the next chapter. The ideas that make this construction possible will play a vital role in the rest of the construction as well.

In Chapter \ref{Not coorientable}, we will consider the general case of a smooth, embedded, possibly non-coorientable submanifold $D\subseteq M\,.$ The method will provide a non-Hausdorff Lie groupoid. We shall also show that the existence of normal covering $\pi:M_N\to M\,,$ such that $D_N:=\pi^{-1}(D)$ is coorientable, is both necessary and sufficient for the existence of a Hausdorff integration of the elliptic tangent bundle of a smooth, embedded divisor. This result will require proving that, if $D$ is not coorientable, but $M$ has a double cover $M_2\to M$ such that $D_2:=\pi^{-1}(D)$ is coorientable, then the Hausdorff Lie groupoid obtained in Chapter \ref{Coorientation double cover} is final in the category of $s$-connected Hausdorff integrations of the elliptic tangent bundle, where this category is described in \cite{GL12}. 

In Chapter \ref{Normal crossing}, we will complete the construction by studying the case where $D$ is a general normal crossing elliptic divisor. We shall also study the existence of Hausdorff integrations in this case, along with providing a description of the final $s$-connected Hausdorff integration in a tubular neighbourhood of any intersection locus of $D\,.$

In Chapter \ref{EllipticIdeal}, we then construct a canonical elliptic ideal on the constructed groupoid.

In Chapter \ref{SympGroupoid}, we will discuss local models for the integration of Poisson structures associated to elliptic symplectic structures.

\subsection*{Acknowledgements} First and foremost, I would like to thank my supervisor, Gil Cavalcanti, for the problem statement and the many fruitful discussions we had on the topic. Moreover, I would like to thank Marius Crainic and Aldo Witte for their helpful remarks, and Camille Laurent-Gengoux for providing the reference \cite{ADH89}.

\section{Preliminaries on ideal Lie algebroids}\label{prel}
In this section, we will introduce a few relevant Lie algebroids, most notably the elliptic tangent bundle. We will be rather pedestrian, for a more detailed discussion, see \cite{Ralph}. The key idea is that, given a manifold $M$ and an ideal sheaf $\mathcal{I}\subseteq C^\infty(-;\mathbb{R})\,,$ we can consider the sheaf $\mathfrak{X}_{\mathcal{I}}\subseteq\mathfrak{X}$ consisting of vector fields preserving $\mathcal{I}\,,$ i.e. 
$$\mathfrak{X}_{\mathcal{I}}(U):=\{X\in\mathfrak{X}(U):X\mathcal{I}\subseteq\mathcal{I}\}\,,$$
and we see that $\mathfrak{X}_{\mathcal{I}}$ is closed under the Lie bracket and under left multiplication by smooth functions (since $\mathcal{I}$ is an ideal). Therefore, if it is also locally finitely generated and locally free, it defines a Lie algebroid $(\mathcal{A}_{\mathcal{I}},\rho,[-,-])\to M\,,$ such that $\Gamma(-;\mathcal{A})=\mathfrak{X}_\mathcal{I}\,,$ $\rho:\mathcal{A}_{\mathcal{I}}\to TM$ is obtained from the inclusion $\mathfrak{X}_\mathcal{I}\hookrightarrow\mathfrak{X}\,,$ and $[-,-]$ is the restriction of the Lie bracket on $\mathfrak{X}$ to $\mathfrak{X}_{\mathcal{I}}\,.$

Likewise, if $\mathcal{I}$ is an ideal sheaf in $C^\infty(-;\mathbb{C})$ such that $\mathfrak{X}_\mathcal{I}\subseteq\Gamma(-;TM^\mathbb{C})$ is locally finitely generated and locally free, we get a complex Lie algebroid $(\mathcal{A}_{\mathcal{I}},\rho,[-,-])\to M\,.$ If $M$ is a complex manifold and $\mathcal{I}\subseteq\mathcal{O}_M\,,$ the sheaf of holomorphic functions, such that $\mathfrak{X}_{\mathcal{I}}\subseteq \Gamma^\hol(-;T^{1,0}M)$ is locally finitely generated and locally free, we get a holomorphic Lie algebroid $(\mathcal{A}_{\mathcal{I}},\rho,[-,-])\to M\,.$
\begin{definition}[Ideal Lie algebroid]
Suppose $\mathcal{I}\subseteq C^\infty(-;\mathbb{R})$ (respectively $C^\infty(-;\mathbb{C})\,;\,\mathcal{O}_M$) is an ideal sheaf such that $\mathfrak{X}_{\mathcal{I}}$ is locally finitely generated and locally free. Then $(\mathcal{A}_{\mathcal{I}},\rho,[-,-])\to M$ is a \emph{real ideal Lie algebroid} (respectively \emph{complex ideal Lie algebroid}; \emph{holomorphic ideal Lie algebroid}).
\end{definition}
\begin{remark}\label{Singular}
If $\mathcal{I}$ is a real ideal sheaf such that $\mathfrak{X}_{\mathcal{I}}$ is locally finitely generated, then $\mathfrak{X}_\mathcal{I}$ defines a singular foliation in the sense of Stefan-Sussmann \cite{Stefan,Sussmann}, even in the case where it is not free.
\end{remark}
\begin{example}
Let $M$ be a manifold and let $L\to M$ be a real (respectively complex; holomorphic) line bundle equipped with a section $\sigma:M\to L$ transverse to $0\,.$ Then $\mathcal{I}_\sigma:=\sigma(\Gamma(L^*))$ defines an ideal sheaf such that $\mathfrak{X}_{\mathcal{I}_\sigma}$ is locally free \cite{Melrose,CG15}. In this case, we call $\mathcal{A}_{\mathcal{I}_\sigma}$ a \emph{real log-tangent bundle} (respectively \emph{complex log-tangent bundle}; \emph{holomorphic log-tangent bundle})\,.
\end{example}
\begin{example}
The $b$-tangent bundle of Melrose \cite{Melrose} is also an ideal Lie algebroid: it can be viewed as a particular case of the above example where $M$ is a manifold with boundary $\partial M\,,$ and $x$ is a function (i.e. a section of the trivial line bundle) such that $x$ vanishes linearly on $\partial M$ and is nonvanishing on the interior of $M\,.$ The $b$-tangent bundle $^bTM\to M$ of $M$ is the real log-tangent bundle associated to the ideal sheaf $\langle x\rangle\,.$
\end{example}
The main Lie algebroid that we will be interested in is the \emph{elliptic tangent bundle}. These will be ideal Lie algebroids induced by \emph{elliptic divisors}:
\begin{definition}[Elliptic tangent bundle]
Let $L\to M$ be a real line bundle and let $\sigma\in\Gamma(L)$ be a section. Then the pair $(L,\sigma)$ is an \emph{elliptic divisor} if $D:=\sigma^{-1}(0)$ is a smooth, embedded, codimension two submanifold such that the normal Hessian $\mathrm{Hess}(\sigma)\in\Gamma(D;\mathrm{Sym}^2\nu^*D\otimes L)$ is nondegenerate\footnote{It can be shown that $L$ is necessarily trivialisable, see \cite{Ralph}, i.e. $L$ is non-canonically equivalent to the trivial line bundle $\underline{\mathbb{R}}\,,$ such that $\sigma$ is just a function. Since $D$ is codimension 2, this function necessarily vanishes to second order, the condition of the normal Hessian being nondegenerate then just means that the Hessian of $\sigma$ has maximal rank.}. The associated \emph{elliptic ideal} $\mathcal{I}_{|D|}:=\sigma(\Gamma(L^*))$ gives rise to a locally free sheaf $\mathfrak{X}_{|D|}$ (see \cite{CG15}) and therefore an ideal Lie algebroid $\mathcal{A}_{|D|}\to M\,,$ which we call the \emph{elliptic tangent bundle} of $\mathcal{I}_{|D|}\,.$
\begin{remark}
When discussing elliptic tangent bundles, we will often leave the elliptic ideal implicit and call $(M,D)$ an \emph{elliptic pair} and call $D$ the elliptic divisor.
\end{remark}
\end{definition}
Such a definition calls for a motivating example:
\begin{example}
Let $\mathcal{A}_{\mathcal{I}_\sigma}\to M$ be a complex log-tangent bundle induced by $(L,\sigma)\,,$ and define $D:=\sigma^{-1}(0)\,.$ Then consider $(L\otimes \overline{L},\sigma\otimes\overline{\sigma})\,,$ and notice that $L\otimes\overline{L}$ has a canonical real structure given by $\iota(z\otimes\overline{w})=w\otimes\overline{z}\,,$ which fixes $\sigma\otimes\overline{\sigma}\,.$ Thus, we can define $\mathrm{Re}(L\otimes\overline{L}):=\{v\in L\otimes\overline{L}:\iota(v)=v\}\,,$ and see that this is a real line bundle over $M\,.$ Then, since $\sigma$ vanishes linearly on $D\,,$ $\sigma\otimes\overline{\sigma}$ has nondegenerate normal Hessian on $D\,,$ so $(\mathrm{Re}(L\otimes\overline{L}),\sigma\otimes\overline{\sigma})$ becomes an elliptic divisor, thus naturally giving rise to an elliptic tangent bundle $\mathcal{A}_{|D|}\to M\,.$ Moreover, it can be shown that $\mathcal{A}_{|D|}\otimes\mathbb{C}\cong\mathcal{A}_{\mathcal{I}_\sigma}\times_{TM^\mathbb{C}}\mathcal{A}_{\mathcal{I}_{\overline{\sigma}}}\,,$ where the fibred product is taken with respect to the anchor maps of the Lie algebroids, see \cite{Ralph}.
\end{example}
The converse of the above example is not true, not every elliptic ideal comes from a complex ideal. We see that if $\mathcal{A}_{\mathcal{I}_{\sigma}}$ is a complex log-divisor induced by $(L,\sigma)$ with $D:=\sigma^{-1}(0)\,,$ then $\nu D\cong L|_{D}\,,$ in particular, $D$ is coorientable. An elliptic divisor need not be coorientable, as is illustrated by the following
\begin{example}\label{KleinBottle}
Let $K$ denote the Klein bottle, consider an embedding $K\hookrightarrow \mathbb{R}^4$ and consider the trivial line bundle $\underline{\mathbb{R}}\to\mathbb{R}^4\,,$ equipped with a section $\sigma\,,$ which is the square distance to $K$ in a tubular neighbourhood, and nonvanishing everywhere else. Then $(\mathbb{R}^4,K)$ defines an elliptic pair, but since $K$ is not coorientable, it is not induced by a complex ideal.
\end{example}

In fact, coorientation is the only obstruction that can occur here. If the elliptic divisor $D$ is coorientable, we see that it always induces a complex ideal in the following way: suppose $(L,\sigma)$ is an elliptic ideal with $D:=\sigma^{-1}(0)\,$. Then by the Morse-Bott lemma, there is a trivialising embedding chart $U$ around $D$ with coordinates $(x_1,\dots,x_{n-2},v_1,v_2)$ such that $D\cap U=\{v_1,v_2=0\}$ and $\sigma=v_1^2+v_2^2\,.$ Then we can formally take $z=v_1+iv_2\,,$ such that locally we have $\sigma=z\overline{z}\,,$ i.e. the elliptic ideal $\mathcal{I}_{|D|}$ gets induced by the complex ideal $\langle z\rangle\,.$ In fact, because of the coorientability of $D\,,$ we can do this in a tubular neighbourhood of $D$ to obtain globally a complex log-divisor $(L',\sigma')$ inducing $(L,\sigma)\,,$ see also \cite{CG15}. Note that we could have also taken $z=v_1-iv_2$ to get $(\overline{L'},\overline{\sigma'})\,.$ Thus, we see that for a coorientable elliptic divisor $D\,,$ there are two canonical complex ideals $\mathcal{I}^\mathbb{C}_D:=\mathcal{I}_{\sigma'}$ and $\mathcal{I}^\mathbb{C}_{\overline{D}}:=\mathcal{I}_{\overline{\sigma'}}$ associated to $\mathcal{I}_{|D|}\,.$
\begin{definition}[Complex primitive]
The complex ideals $\mathcal{I}^\mathbb{C}_D$ and $\mathcal{I}^\mathbb{C}_{\overline{D}}$ are \emph{complex primitives} of the elliptic ideal $\mathcal{I}_{|D|}\,.$\footnote{A priori there's no way of choosing a canonical complex primitive, so the notation $\mathcal{I}^\mathbb{C}_{\overline{D}}$ is simply in place to make explicit that it is the complex conjugate of the other one, it is in no way less canonical. In fact, choosing a coorientation of $D$ is equivalent to picking either of the complex primitives.}
\end{definition}
Moreover, using the same local frame $\{(x_1,\dots,x_{n-2},v_1,v_2\}\,,$ with $\sigma=v_1^2+v_2^2\,,$ we see that the elliptic tangent bundle is locally generated by the vector fields $\{\partial_{x_1},\dots,\partial_{x_{n-2}},r\partial_r,\partial_\theta\}\,,$ where $v_1+iv_2=:re^{i\theta}\,.$ We also see $r\partial_r=v_1\partial_{v_1}+v_2\partial_{v_2}$ and $\partial_\theta=v_1\partial_{v_2}-v_2\partial_{v_1}\,.$

Next, we discuss the relation between holomorphic log-tangent bundles and associated elliptic tangent bundles. Let $\mathcal{A}_{\mathcal{I}_{\sigma}}$ be a holomorphic log-tangent bundle on a complex manifold $(M,J)\,,$ with $D:=\sigma^{-1}(0)\,$. Since $\mathfrak{X}_{\mathcal{I}_{\sigma}}$ is a subsheaf of $\Gamma(-;T^{1,0}M)\,,$ we can use the bundle isomorphism $\varphi:TM\to T^{1,0}M$ given by $v\mapsto v-iJv$ to see that this defines a subsheaf $\mathfrak{X}_{|D|}\subseteq\mathfrak{X}(M)\,.$ In local holomorphic coordinates $(z_1,\dots,z_n)$ on $U$ around $D\,,$ such that $D\cap U=\{z_n=0\}\,$, we see that $\mathfrak{X}_{I_{\sigma}}$ is generated by the holomorphic vector fields $\{\partial_{z_1},\dots,\partial_{z_{n-1}},z_n\partial_{z_n}\}\,.$ Letting $z_i=x_i+iy_i\,,$ we see that $\varphi(\partial_{x_i})=2\partial_{z_i}$ and $\varphi(\partial_{y_i})=2i\partial_{z_i}\,.$ Therefore, we see that $\mathfrak{X}_{|D|}$ is generated by $\{\partial_{x_1},\partial_{y_1},\dots,\partial_{x_{n-1}},\partial_{y_{n-1}},x_n\partial_{x_n}+y_n\partial_{y_n},x_n\partial_{y_n}-y_n\partial_{x_n}\}\,.$ Thus, $\mathfrak{X}_{|D|}$ corresponds to the elliptic tangent bundle $\mathcal{A}_{|D|}$ that is induced by the ideal $\langle x_n^2+y_n^2\rangle$ in the above coordinates.

The above construction generalises to arbitrary holomorphic Lie algebroids in the following way:
\begin{lemma}[\cite{LSX07}]\label{HoloToReal}
Let $\mathcal{A}\to M$ be a holomorphic Lie algebroid. Consider the underlying vector bundle $\mathcal{A}^\mathbb{R}\to M^\mathbb{R}$ obtained by forgetting the complex structures on $\mathcal{A}$ and on $M\,$. Then under the isomorphism $TM\xrightarrow{\varphi}T^{1,0}M\,,$ $\mathcal{A}^\mathbb{R}$ becomes a real Lie algebroid. Moreover, $\mathcal{A}$ and $\overline{\mathcal{A}}$ form a matched pair of complex Lie algebroids (see \cite{Mokri_1997}), such that $\mathcal{A}\bowtie\overline{\mathcal{A}}$ is the complexification of $\mathcal{A}^\mathbb{R}\,.$
\end{lemma}

Finally, we note that the definition of elliptic tangent bundle still makes sense if $D$ is a \emph{normal crossing elliptic divisor}. That is, $D$ is no longer assumed to be smooth, but rather locally looks like a collection of codimension two coordinate hyperplanes. 
\begin{definition}
A \emph{normal crossing elliptic divisor} is a divisor $D$  with elliptic ideal $\mathcal{I}_{|D|}\,,$ such that  every $x\in M$ has a coordinate chart $U$ with coordinates $(x_1,\dots,x_{n-2k},z_1,\dots,z_k)$ with $z_j\in\mathbb{C}\,,$ such that $D\cap U=\{\prod_j z_j=0\}$ and $\mathcal{I}_{|D|}=(U)=\langle\prod_j z_j\overline{z_j}\rangle\,.$ This defines an ideal Lie algebroid (see \cite{Aldo}), which we call the \emph{elliptic tangent bundle} of $D$ and which we denote by $\mathcal{A}_{|D|}\to M\,.$
\end{definition}
In particular, if we go to local coordinates, we see that $D$ looks like the intersection of $k$ hyperplanes $H_j:=\{z_j=0\}$ equipped with elliptic ideals $\mathcal{I}_{|H_j|}:=\langle z_j\overline{z_j}\rangle\,,$ satisfying $\mathcal{I}_{|D|}=\prod_j\mathcal{I}_{|H_j|}\,.$ Locally, we then get a decomposition $\mathcal{A}_{|D|}|_U\cong\mathcal{A}_{|H_1|}\times_{TM}\dots\times_{TM}\mathcal{A}_{|H_k|}\,,$ as shown in \cite{Aldo}.

\section{Integration case I: $D$ is smooth and coorientable}\label{Coorientable}
In the following, $M$ will be a closed manifold and $D\subseteq M$ will be a smooth, embedded, coorientable, codimension two submanifold equipped with an elliptic ideal $\mathcal{I}_{|D|}\,,$ inducing elliptic tangent bundle $\mathcal{A}_{|D|}\to M\,.$ Following \cite{GL12} and \cite{GLP18}, we will try to find an integration of the associated elliptic tangent bundle by doing a complex blow-up of $D\times D$ inside the pair groupoid $M\times M\,,$ and then removing certain arrows.

To do a complex blow-up, we will need a holomorphic ideal for $D\times D\,.$ Since $D$ is coorientable, there are two complex primitives that induce $\mathcal{I}_{|D|}\,.$ We pick one and denote it by $\mathcal{I}^\mathbb{C}_D\,.$ Then we can define a holomorphic ideal for $D\times D\subseteq M\times M$ by $\holI_{D\times D}:=\pi_1^*\mathcal{I}^\mathbb{C}_D+\pi_2^*\mathcal{I}^\mathbb{C}_D\,,$ where the pullback $f^*\mathcal{I}$ of an ideal $\mathcal{I}$ along a map $f:N\to M$ is defined as the ideal sheaf generated by the functions $\{\phi\circ f:\phi\in\mathcal{I}\}\,.$

We define the blow-up $\beta:\widetilde{M\times M}\to M\times M$ along $D\times D$ with respect to this holomorphic ideal.
\begin{lemma}
The above construction is independent of chosen complex primitive of $D\,$, that is, if we had picked $\mathcal{I}^\mathbb{C}_{\overline{D}}\,$, the blow-ups are diffeomorphic by a unique diffeomorphism.
\end{lemma}
\begin{proof}
Let $J$ denote the complex structure on $\nu D$ induced by $\mathcal{I}^\mathbb{C}_D\,,$ such that the one induced by $\mathcal{I}^\mathbb{C}_{\overline{D}}$ is $-J\,.$ Then the repective induced complex structures on $\nu(D\times D)\cong \pi_1^*\nu D\oplus \pi_2^*\nu D$ are
$$\begin{pmatrix}
J&0\\ 0&J
\end{pmatrix}\quad\text{and}\quad
\begin{pmatrix}
-J&0\\ 0&-J
\end{pmatrix}\,.
$$
Since these have the same lines in $\nu (D\times D)\,,$ the associated blow-ups are canonically diffeomorphic.
\end{proof}
Since we want to produce an integration of the elliptic tangent bundle, our groupoid must have an orbit over (the connected components of) $D$ and an orbit over (the connected components of) $M\setminus D\,.$ Thus, analogously to the $b$-tangent bundle case in \cite{GL12}, we define 
\begin{equation}\label{ellipticGroupoid}
\mathcal{G}:=\widetilde{M\times M}\setminus\overline{\beta^{-1}[(M\setminus D\times D)\cup (D\times M\setminus D)]}\,,
\end{equation}
and $s,t:\mathcal{G}\to M$ by $s:=\pi_2\circ \beta$ and $t:=\pi_1\circ\beta\,.$ To define the other maps required for a groupoid structure, we see that on $M\setminus D\times M\setminus D\,,$ we can just take the pair groupoid, meaning on $M\setminus D\times M\setminus D\,,$ we have the following maps
\begin{align*}
m((x,y),(y,z))&=(x,z)\\
\iota(x,y)&=(y,x)\\
u(x)&=(x,x)\,.
\end{align*}
\begin{theorem}[Elliptic groupoid I: Coorientable smooth divisor]\label{ellipticGroupoidStructureMaps}
The above maps extend to smooth maps $m:\mathcal{G}\times_M\mathcal{G}\to\mathcal{G}\,,$ $\iota:\mathcal{G}\to\mathcal{G}$ and $u:M\to\mathcal{G}\,.$ In total, the tuple $(\mathcal{G},s,t,m,\iota,u)$ is a Lie groupoid that integrates the elliptic tangent bundle $\mathcal{A}_{|D|}\to M\,$.
\end{theorem}

\begin{proof}
Note that $s$ and $t$ are submersions and that, if all maps extend smoothly from $M\setminus D\times M\setminus D$ to $\mathcal{G}\,,$ then these automatically form a Lie groupoid because the remaining conditions are all closed, and $M\setminus D\times M\setminus D$ is dense in $\mathcal{G}\,.$ So we just have to show that the maps extend smoothly and that the Lie algebroid is indeed the elliptic tangent bundle.

To show that the maps extend, we use a result from \cite{Joey} that tells us that we can pick a tubular neighbourhood $U\cong\nu D$ such that $\mathcal{I}^\mathbb{C}_D|_{U}$ corresponds to $\mathcal{I}^\mathbb{C,\mathrm{lin}}_D$ on $\nu D\,,$ which we equipped with the complex structure induced by $\mathcal{I}^\mathbb{C}_D\,.$ Here, $\mathcal{I}_D^{\mathbb{C},\mathrm{lin}}$ is the complex ideal on $\nu D$ that, given a local trivialisation with coordinate $v\,,$ is given by $\langle v\rangle\,.$
On $\beta^{-1}[U]\,,$ trivialising charts for the normal bundle (equipped with the complex structure) then give rise to coordinate charts $(x_1,\dots,x_{n-2},v)$ on $U\,,$ where $x:=(x_1,\dots,x_{n-2})$ are coordinates on $D\,,$ and $v\in\mathbb {C}$ is the trivialisation of the normal bundle. So, in these coordinates, the multiplication becomes 
$$
(x,v_1,y,v_2)\cdot(y,v_2,z,v_3)=(x,v_1,z,v_2)\,,
$$
which we can use to extend the multiplication smoothly to the blow-up: if $(x,y),(y,z)\in D\times D\,,$ and $[v_1:v_2]$ denotes a line in $\nu_{(x,y)}(D\times D)$ and $[v_3:v_4]$ denotes a line in $\nu_{(y,z)}(D\times D)\,,$ we can interpret them as points on the exceptional divisor of the blow-up $\widetilde{M\times M}\,.$ Since we cut out $\overline{\beta^{-1}[(M\setminus D\times D)\cup (D\times M\setminus D)]}$ from $\widetilde{M\times M}$ to obtain $\mathcal{G}\,,$ we see that $[v_1:v_2]$ and $[v_3:v_4]$ lie in $\mathcal{G}$ if and only if $v_1,v_2,v_3,v_4\neq 0\,.$ Thus, in order to show that multiplication extends to $D\,,$ we have to define the multiplication $(x,y,[v_1:v_2])\cdot(y,z,[v_3:v_4])$ if and only if $v_1,v_2,v_3,v_4\neq 0\,.$ But now we see that these points lie in the following respective blow-up charts:
\begin{align*}
(x,y,a_1,b_1)&\,,\qquad \beta(x,y,a_1,b_1)=(x,a_1,y,a_1b_1)\,;\\
(y,z,a_2,b_2)&\,,\qquad \beta(y,z,a_2,b_2)=(y,a_2,z,a_2b_2)\,,
\end{align*}
where $b_1,b_2\neq 0\,.$ If we then also introduce the following blow-up coordinates around $(x,z)\,:$
$$(x,z,a_3,b_3)\,,\qquad \beta(x,z,a_3,b_3)=(x,a_3,z,a_3b_3)\,,$$
we see that multiplication away from the exceptional divisor (i.e. $a_1\neq 0$) becomes
$$(x,y,a_1,b_1)\cdot (y,z,a_1b_1,b_2)=(x,z,a_1,b_1b_2)\,.$$
It is clear that this extends smoothly to the exceptional divisor (i.e. $a_1=0$)\,, by 
$$(x,y,0,b_1)\cdot(y,z,0,b_2)=(x,z,0,b_1b_2)\,.$$
Thus we get the multiplication $(x,y,[v_1:v_2])\cdot(y,z,[v_3:v_4])=(x,z,[(v_3/v_2)v_1:v_4])\,.$ Using these coordinates, we can also see that $\iota$ extends smoothly to $\iota(x,y,[v_1:v_2])=(y,x,[v_2:v_1])\,,$ and that $u$ extends smoothly to $u(x)=(x,x,[v:v])$ for $x\in D\,.$

To show that the associated Lie algebroid is indeed the elliptic tangent bundle, we again go to these coordinates and compute $\ker(ds)$ in the above blow-up coordinates around $(x,x)\in D\times D\,,$ which we will do on the complexifications of $T\mathcal{G}$ and $TM\,,$ since it will simplify the computations. Since every map is real, this will not change the computations in any way, we can just take the real part in the end. We have
$$s(x_1,\dots,x_{n-2},x'_1,\dots,x'_{n-2},a,b)=(x'_1,\dots,x'_{n-2},ab)\,,$$
so 
$$ds^\mathbb{C}=(dx'_1,\dots,dx'_{n-2},b\,\partial a+a\,\partial b+\barb\,\bpartial \bara+\bara\,\bpartial \barb)\,,$$
where we note that the fibres of the tubular neighbourhood canonically define one-dimensional complex manifolds, so $\partial$ and $\bpartial$ are well defined. This gives us $$\ker(ds^\mathbb{C})=\mathrm{span}_\mathbb{C}\{\partial_{x_1},\dots,\partial_{x_{n-2}},a\partial_a-b\partial_b,\bara\partial_{\bara}-\barb\partial_{\barb}\}\,.$$ 
Applying $dt^\mathbb{C}=(dx_1,\dots,dx_{n-2},\partial a+\bpartial\bara)$ to this, we obtain
$$dt^\mathbb{C}(\ker(ds^\mathbb{C}))=\mathrm{span}_\mathbb{C}(\partial_{x_1},\dots,\partial_{x_{n-2}},v\partial_v,\overline{v}\partial_{\overline{v}})\,,$$
where $v$ is the coordinate on the fibres of the tubular neighbourhood. Taking the real part, we see that we get
$$dt(\ker(ds))=\mathrm{span}_{\mathbb{R}}(\partial_{x_1},\dots,\partial_{x_{n-2}},r\partial_r,\partial_\theta)\,,$$
where $v=re^{i\theta}\,,$ so we indeed see that this is the elliptic tangent bundle.
\end{proof}

\begin{remark}
If either $M$ or $D$ is disconnected, this construction will give a groupoid that is not $s$-connected, i.e. with disconnected $s$-fibres. To illustrate this, take $\mathbb{C}P^1$ with complex divisor $\{0,\infty\}\,,$ and integrate the elliptic tangent bundle as above. Then it will have two orbits: $\mathbb{C}^*$ and $\{0,\infty\}\,.$ Since $\{0,\infty\}$ is a disconnected set, this is not $s$-connected. We will denote the $s$-connected component of $\mathcal{G}$ by $\mathcal{G}_0\,.$
\end{remark}

\begin{definition}[Elliptic groupoid I: Coorientable smooth divisor]
The groupoid $(\mathcal{G}_0,s,t,m,\iota,u)$ is the \emph{elliptic groupoid} of the pair $(M,D)\,,$ where $\mathcal{G}_0$ is the $s$-connected component of $\mathcal{G}$ from Equation \eqref{ellipticGroupoid}, with the structure maps defined in Lemma \ref{ellipticGroupoidStructureMaps}.
\end{definition}
It follows from \cite{GLP18} that this is the final groupoid in the category of $s$-connected integrations of the elliptic tangent bundle $\mathcal{A}_{|D|}\,.$

The following is evident from the above discussion:
\begin{proposition}
The elliptic groupoid $\mathcal{G}$ of $(M,D)$ has the following properties:
\begin{enumerate}
    \item The orbits of $\mathcal{G}$ are the connected components of $M\setminus D$ and $D\,$;
    \item The isotropy of $\mathcal{G}$ over $M\setminus D$ is trivial;
    \item The isotropy of $\mathcal{G}$ over $D$ is $\mathbb{C}^*\,.$
\end{enumerate}
\end{proposition}

\begin{remark}\label{Holomorphic}
There is a strong relation between our construction and the one done in \cite{GLP18}, where they describe the holomorphic integration of the holomorphic log tangent bundle of a divisor $D$ in a complex manifold $M\,.$ The groupoid we obtain above is naturally diffeomorphic to the underlying real space of this holomorphic groupoid. This symmetry comes from a general fact: as mentioned in Lemma \ref{HoloToReal}, there is a canonical way of associating an underlying real Lie algebroid to a complex Lie algebroid. As shown in \cite{LSX07,LSX08}, this also lifts to the level of groupoids, i.e. an integration of the holomorphic Lie algebroid to a holomorphic Lie groupoid becomes an integration of the underlying real Lie algebroid to a real Lie groupoid by forgetting the complex structures. Likewise, the source-simply connected integration of the underlying real Lie algebroid always admits a complex structure that integrates the holomorphic Lie algebroid.

If we let $\mathrm{LieAlgebroid}^\hol(M)$ denote the category of holomorphic Lie algebroids over $M\,,$ whose morphisms are holomorphic Lie algebroids morphisms covering $\id_M\,,$ $\mathrm{LieAlgebroid}^\mathbb{R}(M)$ denote the category of real Lie algebroids over $M\,,$ whose morphisms cover $\id_M\,,$ and likewise $\mathrm{LieGroupoid}^\hol(M)$ and $\mathrm{LieGroupoid}^\mathbb{R}(M)$ denote the categories of (holomorphic-) Lie groupoids over $M$ whose morphisms cover $\id_M\,,$ this construction gives a square of functors commuting up to natural isomorphism
\[\begin{tikzcd}
	{\mathrm{LieGroupoid}^\hol(M)} & {\mathrm{LieGroupoid}^\mathbb{R}(M)} \\
	{\mathrm{LieAlgebroid}^\hol(M)} & {\mathrm{LieAlgebroid}^\mathbb{R}(M)\,,}
	\arrow[from=1-1, to=1-2]
	\arrow[from=1-1, to=2-1]
	\arrow[from=1-2, to=2-2]
	\arrow[from=2-1, to=2-2]
\end{tikzcd}\]
where the top horizontal arrow is a forgetful functor, the bottom horizontal arrow is the construction $\mathcal{A}^\hol\mapsto\mathcal{A}\,,$ and the vertical arrows are the Lie functors in the appropriate categories.
\end{remark}

\begin{remark}
In \cite{GLP18}, it is shown that in the case of a compact Riemann surface $X$ (without boundary), equipped with a divisor $D=\{p_1,\dots,p_k\}\,,$ the $s$-simply connected groupoid $\Pi_1(X,D)$ of $\mathcal{A}_{|D|}\to X$ is Hausdorff unless $X=\mathbb{C}P^1$ and $D$ consists of a single point. This result carries over to the case where $X$ is any compact real surface (without boundary), not necessarily orientable. 

To see this, note that the $s$-simply connected groupoid of $\mathcal{A}_{|\{0\}|}\to\mathbb{C}\,,$ is $\mathbb{C}^2$ with $s(v_1,v_2)=v_2\,,$ $t(v_1,v_2)=v_2e^{v_1}\,,$ and multiplication $(v_1,v_2e^{v_3})\cdot(v_3,v_2)=(v_1+v_3,v_2)\,,$ see also \cite{ADH89}. If we restrict this to $\mathbb{C}\setminus \{0\}\,,$ this is just a way of representing the fundamental groupoid $\Pi_1(\mathbb{C}\setminus\{0\})\,.$

If we then pick smooth coordinates $U\cong\mathbb{C}$ centred at some $p\in D\,,$ such that $D\cap U=\{p\}\,,$ we see that $\Pi_1(X,D)|_{U}$ is Hausdorff if and only if the pushforward map 
$$\iota_*:\Pi_1(U\setminus \{p\})\to \Pi_1(X,D)|_{U\setminus \{p\}}\cong\Pi_1(X\setminus D)|_{U\setminus \{p\}}$$ 
is injective, which can be checked at the level of isotropies. Using the classification of surfaces, we know that the pushforward $\pi_1(U\setminus \{p\})\xrightarrow{\iota_*}\pi_1(X\setminus D)$ is injective with the sole exception $X=\mathbb{C}P^1$ and $D=\{p\}\,,$ in which case $\iota_*$ is the zero map.
\end{remark}

\begin{remark}\label{ASHolonomyGroupoid}
By Remark \ref{Singular}, $\mathfrak{X}_{|D|}$ is a singular foliation in the sense of Stefan-Sussmann \cite{Stefan,Sussmann}. The groupoid that we construct above is then the holonomy groupoid of this singular foliation \cite{Debord,AS07}.
\end{remark}

\section{Integration case II: $(M,D)$ has a coorientation normal covering}\label{Coorientation double cover}
\subsection{$(M,D)$ has a coorientation double cover}
Suppose that $(M,D)$ has a double cover $\pi:M_2\to M$ such that $D_2:=\pi^{-1}(D)$ is coorientable in $M_2\,.$ In this case, we can find the elliptic groupoid $\mathcal{G}_2\rightrightarrows M_2$ integrating $\mathcal{A}_{|D_2|}\to M_2\,,$ using the construction of Chapter \ref{Coorientable}. Here, note that $(M_2,D_2)$ inherits the structure of an elliptic pair by pulling back the elliptic ideal of $(M,D)$ along $\pi\,.$ 

Since it is a double cover, $(M_2,D_2)$ has a free $\mathbb{Z}/2$ action on it by deck transformations, which acts transitively on the fibres of $\pi\,.$ The idea is to lift this $\mathbb{Z}/2$ action to $\mathcal{G}_2\,,$ making all groupoid maps equivariant, and then quotienting out this $\mathbb{Z}/2$ action to obtain a groupoid over $M$ that integrates the elliptic tangent bundle $\mathcal{A}_{|D|}\to M\,$. Note that from here on, we define $\mathbb{Z}/2=\{1,\alpha\}\,,$ with $\alpha^2=1\,.$
\begin{lemma}
The $\mathbb{Z}/2$ action on $(M_2,D_2)$ lifts to $\mathcal{G}_2\,,$ such that all structure maps become equivariant.
\end{lemma}
\begin{proof}
For $(x,y)\in M_2\setminus D_2\times M_2\setminus D_2\,,$ we take $\alpha(x,y)=(\alpha(x),\alpha(y))\,.$ Now fix a tubular neighbourhood $U$ of $D$ and consider $U_2:=\pi^{-1}(U)\subseteq M_2\,.$ Then $U_2\times U_2$ is naturally a tubular neighbourhood of $D_2\times D_2$ in $M_2\,.$ For $(x,y,\ell)\in\beta^{-1}(D\times D)\,,$ we take $\alpha(x,y,\ell)=(\alpha(x),\alpha(y),\alpha(\ell))\,,$ where $\ell$ denotes the line through $(v,w)$ inside the fibre of $U_2\times U_2$ at the point $(x,y)\,,$ and $\alpha(\ell)$ is the line through $(\alpha(v),\alpha(w))\,,$ which now lies in the fibre at $(\alpha(x),\alpha(y))\,.$ It is now clear that all structure maps become equivariant.
\end{proof}
Thus, we get a well defined quotient map $\mathcal{G}_2\to\mathcal{G}_2/(\mathbb{Z}/2)\,,$ which comes equipped with two surjective submersions $s:=\pi\circ s_2$ and $t:=\pi\circ t_2$ to $M\,.$ However, it is not yet a groupoid, precisely since $\mathcal{G}_2/(\mathbb{Z}/2)\times_M \mathcal{G}_2/(\mathbb{Z}/2)$ is bigger than $(\mathcal{G}_2\times_{M_2}\mathcal{G}_2)/(\mathbb{Z}/2)$ (using the obvious $\mathbb{Z}/2$ action) or, more explicitly, an arrow from $x$ to $y$ in $M_2$ must be composable with an arrow from $\alpha(y)$ to $z$ in $M_2$ if $\mathcal{G}_2/(\mathbb{Z}/2)$ was a groupoid. To fix this, we have the following lemma
\begin{theorem}[Elliptic groupoid IIa: Coorientation double cover version]
The definition of $m_2:\mathcal{G}_2\times_{M_2}\mathcal{G}_2\to\mathcal{G}_2$ extends naturally to a map $m_2:\mathcal{G}_2\times_M\mathcal{G}_2\to\mathcal{G}_2\,,$ where the fibre product is taken with respect to $\pi\circ s_2$ and $\pi\circ t_2\,.$ Moreover, the defined map is smooth, and is $\mathbb{Z}/2$-equivariant, where the $\mathbb{Z}/2$ action on $\mathcal{G}_2\times_M\mathcal{G}_2$ is given by $\alpha(g,h)=(\alpha(g),\alpha(h))\,$. Using this extra data, $\mathcal{G}:=\mathcal{G}_2/(\mathbb{Z}/2)$ becomes a Lie groupoid that integrates the elliptic tangent bundle $\mathcal{A}_{|D|}$.
\end{theorem}
\begin{proof}
Let $g,h\in\mathcal{G}_2$ be such that $s_2(g)=\alpha\circ t_2(h)\,.$ Then we define $m_2(g,h)=m_2(g,\alpha^{-1}(h))\,.$ Smoothness follows from smoothness of $m_2|_{\mathcal{G}_2\times_{M_2}\mathcal{G}_2}$ and $\alpha\,.$ Moreover,
\begin{equation}\label{nonLocalMultiplication}
\alpha(m_2(g,h))=\alpha(m_2(g,\alpha^{-1}(h)))=m_2(\alpha(g),h)=m_2(\alpha(g),\alpha(h))\,,
\end{equation}
so this is indeed $\mathbb{Z}/2$-equivariant. That $\mathcal{G}$ is a Lie groupoid is now clear.
\end{proof}
Thus we define
\begin{definition}[Elliptic groupoid IIa: Coorientation double cover version]\label{CoorientableDoubleCover}
Let $D\subseteq M$ be a non-coorientable elliptic divisor such that a coorientation double cover $(M_2,D_2)$ exists. Then we call $\mathcal{G}$ from the previous Lemma the \emph{elliptic groupoid} associated to the double cover $(M_2,D_2)\to(M,D)\,.$
\end{definition}
Note that the $\mathbb{Z}/2$ action is discrete, and the quotient doesn't affect the multiplication for arrows close to the diagonal. Therefore, the Lie algebroid of $\mathcal{G}$ is indeed the elliptic tangent bundle $\mathcal{A}_{|D|}\,.$

The following is evident from the above discussion:
\begin{proposition}
The elliptic groupoid of $(M,D)$ has the following properties:
\begin{enumerate}
    \item The orbits of $\mathcal{G}$ are the connected components of $M\setminus D$ and $D\,$;
    \item The isotropy of $\mathcal{G}$ over $M\setminus D$ is $\mathbb{Z}/2\,$;
    \item The isotropy of $\mathcal{G}$ over $D$ is $\mathbb{C}^*\rtimes\mathbb{Z}/2\,,$ where $\alpha\cdot\lambda=\blambda\,.$
\end{enumerate}
\end{proposition}

\subsection{$(M,D)$ has a coorientation normal covering}\label{Monodromy}
Suppose that, instead of a double cover, $(M,D)$ has a coorientation normal covering, i.e. a covering $\pi:(M_N,D_N)\to(M,D)$ that corresponds to a normal subgroup $N\subseteq\pi_1(M)\,,$ such that $D_N:=\pi_N^{-1}[D]$ is a (possibly disconnected) coorientable codimension two submanifold. We denote the elliptic groupoid of $(M_N,D_N)$ by $\mathcal{G}_N\,,$ which is $s$-connected by definition. We note that the $\pi_1(M)/N$ action on $(M_N,D_N)$ by deck transformations carries over to $\mathcal{G}_N\,,$ as the proofs of the above results (in particular, Equation \eqref{nonLocalMultiplication}) all carry over to this setting. Therefore, we see that $\mathcal{G}:=\mathcal{G}_N/(\pi_1(M)/N)$ becomes a Hausdorff Lie groupoid over $M\,,$ which integrates the elliptic tangent bundle.
\begin{definition}[Elliptic groupoid IIb: Coorientation normal covering version]
Let $D\subseteq M$ be a non-coorientable codimension two submanifold such that a connected coorientation normal covering $(M_N,D_N)\to (M,D)$ exists. Then we call $\mathcal{G}$ from the above the \emph{elliptic groupoid} associated to the normal covering $(M_N,D_N)\to(M,D)\,.$
\end{definition}

In fact, this is a particular case of a construction in \cite{MM03}, where they construct quotients of Lie groupoids by right actions of other Lie groupoids. In our setting, this works as follows: suppose $\pi:P\to M$ is a discrete principal bundle with discrete group $G\,,$ and let $\mathcal{G}\rightrightarrows P$ be an $s$-connected Lie groupoid such that $G$ acts on $\mathcal{G}$ equivariantly and freely. Then we can take the quotient to obtain a groupoid $\mathcal{G}/G\rightrightarrows M\,.$ We see that $G$ acts canonically on the space of orbits $\mathcal{O}$ of $\mathcal{G}\,,$ by diffeomorphisms. Thus, the space of orbits of $\mathcal{G}/G$ is $\mathcal{O}/G\,,$ where the orbits are represented by $\pi(O)$ for $O\in\mathcal{O}\,.$ 

Let $x\in \pi(O)$ for some orbit $O$ of $\mathcal{G}\,.$ Since $P$ is a discrete principal bundle, we get a monodromy action $\varphi:\pi_1(M,x)\to G\,,$ such that $\pi^{-1}(x)\cap O\cong(\varphi\circ i_*)(\pi_1(\pi(O),x))\,,$ where $i:\pi(O)\hookrightarrow M$ is the inclusion. Thus, letting $H_O$ denote the isotropy of $O\,,$ we see that 
$$H_{\pi(O)}\cong H_O\rtimes(\varphi\circ i_*)(\pi_1(\pi(O)))\,,$$ 
where the action of $(\varphi\circ i_*)(\pi_1(\pi(O)))$ on $H_O$ is induced from the action it has on $\mathcal{G}\,.$

Thus, we have
\begin{proposition}
The elliptic groupoid $\mathcal{G}$ of $(M,D)$ corresponding to the coorientation normal covering $(M_N,D_N)\to(M,D)$ has the following properties:
\begin{enumerate}
    \item The orbits of $\mathcal{G}$ are the connected components of $M\setminus D$ and $D\,$;
    \item The isotropy of $\mathcal{G}$ over $M\setminus D$ is $\pi_1(M)/N\,$;
    \item The isotropy of $\mathcal{G}$ over $D$ is $\mathbb{C}^*\rtimes(\varphi\circ i_*)(\pi_1(D))\,,$ where $\varphi:\pi_1(M)\to\pi_1(M)/N$ is the projection, and the action on $\mathbb{C}^*$ is given by complex conjugation if the loop on $D$ changes coorientation.
\end{enumerate}
\end{proposition}
\begin{remark}
The third point in the above Proposition requires that $\varphi\circ i_*$ does not identify loops that change coorientation with loops that reverse coorientation. This stems from the fact that an obstruction to the existence of coorientation normal coverings $(M_N,D_N)\to (M,D)$ is that $\ker(i_*)$ should not contain any loops changing the coorientation. This idea is vital to the discussion on the existence of Hausdorff integrations in the next chapter.
\end{remark}

Integrations of $TM$ are described by totally disconnected normal Lie subgroupoids of the fundamental groupoid $\Pi_1(M)\,,$ which are pointwise normal subgroups of $\pi_1(M)\,.$ Using the above description, we see that the integration corresponding to some normal subgroup $N\leq\pi_1(M)$ then becomes $(P\times P)/G\,,$ where $P$ is the normal covering, and $G=\pi_1(M)/N\,,$ which acts on $P\times P$ by $(x,y)\cdot g=(xg,yg)\,.$ The monodromy representations $\varphi:\pi_1(M,x)\to G$ induce a map $\Phi:\Pi_1(M)\to (P\times P)/G$ by $\Phi(\gamma)=[\gamma]\,,$ with $[\gamma]=[(x,y)]$ for some $y\in P_{s([\gamma])}$ and $x=[\gamma]\cdot y\,,$ where $[\gamma]$ acts by parallel transport along the homotopy class $\gamma\,.$ We note that this map is a surjective Lie groupoid morphism that is a local diffeomorphism. Therefore, it is the unique map from $\Pi_1(M)$ to $(P\times P)/G$ in the category of integrations of $TM\,.$

\section{Integration case III: $D$ is smooth but not coorientable}\label{Not coorientable}
Now we tackle the slightly more general case where $D$ is a smooth elliptic divisor that is not coorientable. The idea is that in the previous case, we had a global double cover that makes the divisor coorientable, however, here we only have that locally in some tubular neighbourhood. Therefore, the idea is to do the previous construction locally (by explicitly writing down a double cover of the tubular neighbourhood), and then gluing the constructed groupoid into the pair groupoid of $M\setminus D$ in a non-Hausdorff way.

We start off our construction by taking a tubular neighbourhood $U\cong \nu D$ of $D$ such that $\mathcal{I}_{|D|}$ corresponds to $\mathcal{I}^{\mathrm{lin}}_{|D|}$ on $\nu D\,.$ Now we consider the orientation bundle $\pi:o(\nu D)\to D$ of the normal bundle, which is a nontrivial $\mathbb{Z}/2$-principal fibre bundle, and we note that $\pi^*\nu D$ is orientable. We also see that $o(\nu D)$ extends to a bundle over $U$ by pulling back along $U\to D\,,$ defining a double cover $\pi:U_2\to U\,.$ Moreover, since $\nu D$ is the normal bundle to $D$ in $U\,,$ we see that $\pi^*\nu D$ is the normal bundle to $D_2:=o(\nu D)$ in $U_2\,,$ since $\pi$ is a local diffeomorphism. Thus we can construct the elliptic groupoid $\mathcal{G}_U$ of $(U,D)$ using the double cover $(U_2,D_2)\to (U,D)\,.$ 

Now we define $\mathcal{G}$ using the following pushout:
\[\begin{tikzcd}
	{\mathcal{G}_U|_{U\setminus D}} &  {(M\setminus D\times M\setminus D)_0}\\
	{\mathcal{G}_U} & {\mathcal{G}\,,}
	\arrow[hook', from=1-1, to=2-1]
	\arrow[from=1-1, to=1-2]
	\arrow[from=2-1, to=2-2]
	\arrow[from=1-2, to=2-2]
	\arrow["\ulcorner"{anchor=center, pos=0.125}, draw=none, from=1-1, to=2-2]
\end{tikzcd}\]
where the subscript $0$ indicates that we are restricting to the $s$-connected component of $M\setminus D\times M\setminus D\,,$ i.e. $(M\setminus D\times M\setminus D)_0$ is the disjoint union of the pair groupoids of the connected components of $M\setminus D\,.$
We see that $\mathcal{G}$ naturally inherits smooth coordinate charts from $\mathcal{G}_U$ and $M\setminus D\times M\setminus D\,,$ thus making it a non-Hausdorff manifold.

\begin{theorem}[Elliptic groupoid III: Non-coorientable non-Hausdorff version]
$\mathcal{G}$ is a non-Hausdorff manifold that naturally inherits a smooth groupoid structure from $\mathcal{G}_U$ and the pair groupoid $M\setminus D\times M\setminus D\,.$ Moreover, $\mathcal{G}$ integrates the elliptic tangent bundle of $(M,D)\,.$
\end{theorem}
\begin{proof}

Now, we have two orbits, $M\setminus D$ and $D\,,$ on both of which there is a well-defined groupoid multiplication. We see that this is smooth, since around $D\,,$ we can pick the charts we defined around $\mathcal{G}|_D\,,$ on which groupoid multiplication around $D$ looks like groupoid multiplication in $\mathcal{G}|_U\,,$ which is smooth. Moreover, on $M\setminus D\times M\setminus D\,,$ we just have the pair groupoid with exactly the same smooth structure, so that is also smooth. 

Lastly, in small enough charts around $u(D)\,,$ $\mathcal{G}$ looks like $\mathcal{G}_U\,,$ so around $D\,,$ the Lie algebroid is the elliptic tangent bundle. On $M\setminus D\,,$ the Lie algebroid is $T(M\setminus D)\,,$ which is just $\mathcal{A}_{|D|}|_{M\setminus D}\,.$ Thus, $\mathcal{G}$ is indeed an integration of $\mathcal{A}_{|D|}\,.$ This completes the proof.
\end{proof}

\begin{definition}[Elliptic groupoid III: Non-coorientable non-Hausdorff version]\label{NonHausdorff}
The groupoid $\mathcal{G}$ constructed above is the \emph{non-Hausdorff elliptic groupoid} of the pair $(M,D)\,.$
\end{definition}

The following is evident from the above discussion:
\begin{proposition}
The non-Hausdorff elliptic groupoid of $(M,D)$ has the following properties:
\begin{enumerate}
    \item The orbits of $\mathcal{G}$ are the connected components of $M\setminus D$ and $D\,$;
    \item The isotropy of $\mathcal{G}$ over $M\setminus D$ is trivial;
    \item The isotropy of $\mathcal{G}$ over $D$ is $\mathbb{C}^*\rtimes\mathbb{Z}/2\,,$ where $\alpha\cdot\lambda:=\blambda\,.$
\end{enumerate}
\end{proposition}

\begin{remark}
If $D$ is a non-coorientable elliptic divisor, one can always construct the non-Hausdorff elliptic groupoid of $D\,,$ but the elliptic groupoid coming from a coorientation double cover might not necessarily exist. The obstruction is readily seen to be in $H^1(D;\mathbb{Z}/2)/(i^*H^1(M;\mathbb{Z}/2))\,,$ where $i:D\hookrightarrow M\,,$ because $(M,D)$ admitting a coorientation double cover means that the coorientation double cover $(U_2,D_2)\to (U,D)$ must extend to the entirety of $M\,.$ Here, note that $(U_2,D_2)$ is the unique coorientation double cover of $(U,D)\,.$ Thus, $M$ must have a double cover restricting to $U_2$ on $U\,,$ which can only happen if the fundamental class of $U_2\to U$ in $H^1(U;\mathbb{Z}/2)\cong H^1(D;\mathbb{Z}/2)$ lies in the image of $i^*:H^1(M;\mathbb{Z}/2)\to H^1(D;\mathbb{Z}/2)\,.$ In particular, if $i^*$ is surjective, then a coorientation double cover $(M_2,D_2)\to (M,D)$ exists.
\end{remark}
Now we will turn to studying the problem of existence of Hausdorff integrations. The fundamental idea here is that the main obstruction to Hausdorffness of the elliptic groupoid comes from coorientation reversing loops on $D\,.$ If a coorientation reversing loop on $D$ is nullhomotopic in $M\,,$ there's certainly no covering of $M$ that coorients $D\,.$ In fact, it turns out that this is the only obstruction to the existence of a Hausdorff integration: if all coorientation reversing loops on $D$ are not nullhomotopic in $M\,,$ we have a Hausdorff integration! Written down a bit more topologically flavoured, this becomes:

\begin{theorem}\label{ExistenceOfHausdorffIntegration}
The elliptic tangent bundle $\mathcal{A}_{|D|}\to M$ has a Hausdorff integration if and only if the fundamental class $\eta$ of the orientation bundle of $\nu D\,,$ viewed as a map $\eta:H_1(D;\mathbb{Z})\to\mathbb{Z}/2\,,$ factors through $i_*H_1(D;\mathbb{Z})\subseteq H_1(M;\mathbb{Z})\,.$
\end{theorem}
Firstly, note that if $D$ is coorientable, we always have a Hausdorff integration and the fundamental class of $\nu D$ vanishes, so there this theorem is true. So for the next part, suppose $D$ is not coorientable.

By the Remark above, if the fundamental class of $\nu D$ lies inside $i^*H^1(M;\mathbb{Z}/2)\,,$ we can construct a coorientation double cover $(M_2,D_2)\to(M,D)\,,$ giving us a Hausdorff integration of the elliptic tangent bundle. Moreover, that means that $\eta$ defines a map $H_1(M;\mathbb{Z})\to\mathbb{Z}/2\,,$ so the above theorem definitely holds. If $\eta$ extends to a map $i_*H_1(D;\mathbb{Z})\to\mathbb{Z}/2\,,$ but not to a map from $H_1(M;\mathbb{Z})\,,$ then we know that there is no coorientation double cover, so we will have to do some more work to find a Hausdorff integration. Lastly, if $\eta$ does not factor through $i_*H_1(D;\mathbb{Z})\,$, then we have to show that there is no Hausdorff integration. 

We will first study integrations of the elliptic tangent bundle on the tubular neighbourhood $U\to D\,.$ We have the following: 
\begin{proposition}\label{FinalHausdorff}
Let $D\subseteq M$ be a non-coorientable elliptic divisor and let $U$ be a tubular neighbourhood of $D\,$. Let $\mathcal{G}^{\mathrm{nH}}_U$ denote the non-Hausdorff elliptic groupoid of $\mathcal{A}_{|D|}\to U\,,$ and let $\mathcal{G}^\mathrm{H}_{U}$ denote the Hausdorff integration coming from the coorientation double cover $U_2\to U\,.$ If $\mathcal{G}$ is any $s$-connected integration of $\mathcal{A}_{|D|}\to U\,,$ then there is a surjective groupoid morphism $\mathcal{G}\to\mathcal{G}^{\mathrm{nH}}_U$ in the category of integrations. If $\mathcal{G}$ is also Hausdorff, then there is a surjective groupoid morphism $\mathcal{G}\to\mathcal{G}_U^{\mathrm{H}}$ in the category of integrations. Moreover, these morphisms are unique, i.e. $\mathcal{G}^\mathrm{nH}_U$ and $\mathcal{G}^\mathrm{H}_U$ are final in, respectively, the categories of $s$-connected integrations and the category of $s$-connected Hausdorff integrations of $\mathcal{A}_{|D|}\to U\,$.
\end{proposition}
\begin{proof}
The fact that $\mathcal{G}^\mathrm{nH}_{U}$ is the final $s$-connected integration follows from Proposition 3.8 in \cite{AS07}, noting that our Lie algebroid in fact comes from a singular foliation $\mathfrak{X}_{|D|}\,.$ Hence, the task at hand is proving the finality of $\mathcal{G}_U^\mathrm{H}$ as an $s$-connected Hausdorff integration. Thus, suppose $\mathcal{G}$ is an $s$-connected Hausdorff integration. We will prove that there is a unique surjective Lie groupoid morphism into $\mathcal{G}_U^\mathrm{H}\,.$

By the general theory of Lie algebroid integrations, we know that if such a surjective morphisms exists, it must necessarily be a local diffeomorphism, hence Hausdorffness implies it must be unique. So it remains to show that this morphism exists.

We define $\mathcal{G}_{D}:=\mathcal{G}_U^{\mathrm{H}}|_D\,,$ and note that the isotropy of $\mathcal{G}_D$ is $G:=\mathbb{C}^*\rtimes \mathbb{Z}/2\,.$ Suppose $N$ is a discrete normal subgroup of $G$ and that $(v,\alpha)\in N\,.$ Since $N$ is normal, we know that this means that $(\blambda^{-1},\alpha)\cdot(v,\alpha)\cdot(\lambda,\alpha)\in N$ for any $\lambda\in\mathbb{C}^*\,.$ Thus, $(\bar{v}\blambda^{-1}\lambda,\alpha)\in N$ for any $\lambda\in\mathbb{C}^*\,,$ but that contradicts the hypothesis that $N$ is discrete. Therefore, we know that $N\subseteq\mathbb{C}^*\times\{1\}\,.$ 

Now, let $\mathcal{G}^{\mathrm{ssc}}_U$ denote the source-simply connected integration of the elliptic tangent bundle of $(U,D)\,,$ then we know from \cite{MM03} that there is a unique surjective groupoid morphism $p:\mathcal{G}^{\mathrm{ssc}}_U\to\mathcal{G}_U^{\mathrm{H}}\,,$ that is a local diffeomorphism. Moreover, there is a discrete, totally disconnected, normal Lie subgroupoid $\mathcal{N}$ such that $\mathcal{G}_U^{\mathrm{H}}\cong\mathcal{G}^{\mathrm{ssc}}_U/\mathcal{N}\,,$ such that $p$ is precisely the associated projection, see \cite{GL12} for details. Since this map is a local diffeomorphism, we see that it induces a surjection at the level of isotropies, so letting $G^\mathrm{ssc}$ denote the isotropy of $\mathcal{G}^{\mathrm{ssc}}_U|_D\,,$ we see that normal subgroups of $G^\mathrm{ssc}$ get mapped to normal subgroups of $G\,.$ Now, $\mathcal{G}$ is an $s$-connected integration of the elliptic tangent bundle, so it is the quotient of $\mathcal{G}^\mathrm{ssc}_U$ by a discrete, totally disconnected, normal Lie subgroupoid $\mathcal{N}'\,,$ for which we denote the associated projection by $p':\mathcal{G}^\mathrm{ssc}_U\to\mathcal{G}\,$. In particular, we get a map $\mathcal{G}\to\mathcal{G}^{\mathrm{H}}_U/p(\mathcal{N}')\,.$ 

We claim that $p(\mathcal{N}')=u(U)\,,$ which would then imply that $\mathcal{G}^\mathrm{H}_U$ is indeed final for Hausdorff integrations. Suppose $p(\mathcal{N}')$ is not $u(U)\,,$ then there is a $g\in p(\mathcal{N}')|_{U\setminus D}\,,$ such that $\lim_{t\to 0}tg=(1,\alpha)$ at the point $\pi(s(g))\in D\,,$ where multiplication by $t$ is induced by scalar multiplication in the tubular neighbourhood $\pi:U\to D$ by $t\,,$ where we note that the isotropy of $\mathcal{G}_U^\mathrm{H}$ over $U\setminus D$ is discrete, so $tg$ is well defined. We note that $(1,\alpha)\notin p(\mathcal{N}')\,,$ since then it cannot be normal. Since $\mathcal{G}$ is Hausdorff, $\mathcal{N}'$ is closed, see \cite{HL21}, therefore no lift of $tg$ to $\mathcal{N}'$ can converge in $\mathcal{G}_U^\mathrm{ssc}\,.$ 

However, $\mathcal{G}_U^\mathrm{ssc}$ is isomorphic to the groupoid of $\mca_{|D|}$-paths up to $\mca_{|D|}$-path homotopy, see \cite{CF03}, so we can pick a representative of some lift $\widetilde{g}\in\mathcal{N}'$ of $g\,,$ which is then just a map $\gamma:S^1\to U\setminus D$ representing the homotopy class $\widetilde{g}\in\pi_1(U\setminus D,s(g))\,.$ But now since $\mathcal{N}'$ is closed, we see that the paths $t\gamma$ for $t\in (0,1]\,,$ defined by scalar multiplication in $\nu_D\,,$ converge to some $\mca_U$-path over $D$ as $t\to 0\,$, since (in local coordinates)
\begin{align*}
\dot{(t\gamma)}(\tau)
&=\dot{x}_1(\tau)\partial_1+\dots+\dot{x}_{2n-2}(\tau)\partial_{2n-2}+
\frac{d}{d\tau}(\log(tr(\tau)))tr(\tau)\partial_r
+\dot{\theta}(\tau)\partial_\theta\\
&=\dot{x}_1(\tau)\partial_1+\dots+\dot{x}_{2n-2}(\tau)\partial_{2n-2}+
\frac{d}{d\tau}(\log(r(\tau)))tr(\tau)\partial_r
+\dot{\theta}(\tau)\partial_\theta\,,
\end{align*}
which converges to the $\mca_U$-path given by 
\begin{align*}
a(\tau)&=(\dot{x}_1(\tau),\dots,\dot{x}_{2n-2}(\tau),d\log(r(\tau))/d\tau,\dot{\theta}(\tau))\\
\pi(a)(\tau)&=(x_1(\tau),\dots,x_{2n-2}(\tau),0,0)\,,
\end{align*}
as $t\to 0\,$. Since $t\gamma$ are representatives of some lift of $tg$ to $\mathcal{N}'\,,$ we have a contradiction. Therefore, we see that $\mathcal{N}'\subseteq \mathcal{N}\,,$ and thus that $\mathcal{G}_U^\mathrm{H}$ is the final Hausdorff integration of $\mca_{|D|}\to U\,.$
\end{proof}

The following is a direct consequence of Proposition 3.8 in \cite{AS07}:
\begin{corollary}\label{FinalNonHausdorff}
The non-Hausdorff elliptic groupoid $\mathcal{G}\rightrightarrows M$ of $\mathcal{A}_{|D|}\to M$ is final in the category of $s$-connected integrations of $\mathcal{A}_{|D|}\,.$
\end{corollary}

Now we can prove Theorem \ref{ExistenceOfHausdorffIntegration}.

\begin{proof}[Proof of Theorem \ref{ExistenceOfHausdorffIntegration}]
Suppose $\mathcal{G}$ is a Hausdorff integration. Then $\mathcal{G}|_U$ is a Hausdorff integration in some tubular neighbourhood of $D\,,$ though possibly no longer source connected. But we can consider the global source-simply connected integration $\mathcal{G}^\mathrm{ssc}\,,$ which then gives us a surjective local diffeomorphism $p:\mathcal{G}^\mathrm{ssc}\to \mathcal{G}\,.$ After restricting $\mathcal{G}^\mathrm{ssc}$ to $U$ and further restricting to the source-component of $u(U)\,,$ which we denote by $\mathcal{G}^\mathrm{ssc}|_{U,0}\,,$ this gives us a map $p_{U,0}:\mathcal{G}^\mathrm{ssc}|_{U,0}\to\mathcal{G}|_{U,0}\,.$ But this is now an $s$-connected Hausdorff integration over $U\,,$ so we get a map into $\mathcal{G}^\mathrm{H}_U\,.$ In total, this gives $p':\mathcal{G}^\mathrm{ssc}|_{U,0}\to\mathcal{G}^\mathrm{H}_U\,.$ 

The isotropies over $U\setminus D$ give a map $\widetilde{\varphi}:j_*\pi_1(U\setminus D)\to \mathbb{Z}/2\,,$ where $j:U\setminus D\hookrightarrow M\setminus D\,,$ so because $\mathbb{Z}/2$ is abelian, we see that it descends to a map $\varphi:j_*H_1(U\setminus D;\mathbb{Z})\to \mathbb{Z}/2\,.$ Now, by construction, $o(\nu D)$ was the coorientation double cover of $U\,,$ so the associated index-two subgroup of $\pi_1(U)\cong \pi_1(D)$ consists precisely of loops on $D$ that preserve the coorientation. Moreover, if $[\gamma]\in \pi_1(U\setminus D)$ is such that $k_*[\gamma]\in\pi_1(D)$ preserves coorientation, where $k:U\setminus D\hookrightarrow U\,,$ then the unique map $p'':\mathcal{G}^\mathrm{ssc}_U\to \mathcal{G}^\mathrm{H}_U$ must map $[\gamma]$ to $1\in\mathbb{Z}/2\,,$ and if $k_*[\gamma]$ changes coorientation, then $[\gamma]$ gets mapped to $\alpha\in\mathbb{Z}/2\,,$ which follows from the discussion at the end of Subsection \ref{Monodromy}. In particular, this gives a map $H_1(U\setminus D;\mathbb{Z})\to \mathbb{Z}/2$ that agrees with $\eta\circ k_*\,.$ So in total, since $p''$ factors through $p'$ as $\mathcal{G}^\mathrm{ssc}_U\to \mathcal{G}^\mathrm{ssc}|_{U,0}\to \mathcal{G}_U^\mathrm{H}\,,$ we see that $\varphi\circ j_*=\eta\circ k_*$ as maps $H_1(U\setminus D;\mathbb{Z})\to \mathbb{Z}/2\,.$

Thus, since the following is a pushout diagram
\[\begin{tikzcd}
	{H_1(U\setminus D;\mathbb{Z})} & {j_*H_1(U\setminus D;\mathbb{Z})} \\
	{H_1(D;\mathbb{Z})} & {i_*H_1(D;\mathbb{Z})\,,}
	\arrow["{k_*}"', from=1-1, to=2-1]
	\arrow["{j_*}", from=1-1, to=1-2]
	\arrow[from=1-2, to=2-2]
	\arrow[from=2-1, to=2-2]
\end{tikzcd}\]
which follows from the Mayer-Vietoris sequence of the cover $\{U,M\setminus D\}\,,$ we see that $\eta$ factors through a map $i_*H_1(D;\mathbb{Z})\to\mathbb{Z}/2\,.$

Conversely, if $\eta$ factors like this, we see that all coorientation changing cycles on $D$ do not get killed by $i_*\,,$ since these are precisely the cycles $[\gamma]$ such that $\eta([\gamma])= \alpha\,.$ Therefore, we can look at the group $h^{-1}(i_*\ker(\eta))\,,$ which is a normal subgroup of $\pi_1(M)$ that contains no loops on $D$ that change the coorientation, where $h:\pi_1(M)\to H_1(M;\mathbb{Z})$ is the Hurewicz map. Therefore, taking the cover associated to this normal subgroup, we get a $\pi_1(M)/h^{-1}(i_*\ker(\eta))$-principal bundle $\pi:P\to M\,,$ such that $\pi^{-1}(D)$ is coorientable, as all loops in $\pi_1(\pi^{-1}(D))$ come from classes in the kernel of $\eta\,,$ which all preserve the coorientation. By the discussion in Subsection \ref{Monodromy}, this gives us a Hausdorff integration of the elliptic tangent bundle.
\end{proof}

\begin{example}
Let $\gamma:S^1\to\mathbb{T}^4$ be a loop that wraps twice around one of the $S^1$-factors of $\mathbb{T}^4\,,$ suppose that $\gamma$ is an embedding and that the $S^1$ factor represents $\xi\in H_1(\mathbb{T}^4;\mathbb{Z})\,,$ i.e. $[\gamma]=2\xi\,$, see Figure \ref{fig:Torus} below. We can pick a tubular neighbourhood $U\to\gamma\,,$ which is the trivial rank three vector bundle over $S^1\,,$ and therefore factors as $U\cong E_2\oplus E_1\,,$ where $E_k$ is the nontrivial rank $k$ vector bundle over $S^1\,.$ Then we can take the embedding $K\to E_2\,,$ of the Klein bottle as the circle bundle, thus giving us an embedding $i:K\hookrightarrow \mathbb{T}^4\,.$ We note that $\gamma_*:H^1(S^1;\mathbb{Z})\to H^1(\mathbb{T}^4)$ is an injection, so $i_*$ does not kill the base circle of $K\,,$ i.e. the coorientation reversing loop. Thus, after picking an elliptic ideal $\mathcal{I}_{|K|}\,,$ e.g. as in Example \ref{KleinBottle}, we obtain that the elliptic tangent bundle $\mathcal{A}_{|K|}\to\mathbb{T}^4$ has a Hausdorff integration. This can also be seen by noting that the four-fold cover along the cycle $\xi$ gives us two embedded $\mathbb{T}^2$'s as a divisor which are coorientable, as seen in Figure \ref{fig:Torus} below. However, the fundamental class $[o(\nu K)]\in H^1(K;\mathbb{Z}/2)$ does not lie in $i^*H^1(\mathbb{T}^4;\mathbb{Z}/2)\,,$ as this would mean that there is a map $\eta:H_1(\mathbb{T}^4;\mathbb{Z})\to\mathbb{Z}/2$ such that $\alpha=\eta([\gamma])=2\eta(\xi)=1\,,$ which cannot be, thus there is no coorientation double cover.\\
\begin{figure}[h!]
\begin{center}
\begin{tikzpicture}
\draw[black,ultra thick] (0,0)--(0,3);
\draw[black,ultra thick] (3,0)--(3,3);
\draw[black,ultra thick] (0,0)--(3,0);
\draw[black,ultra thick] (0,3)--(3,3);
\draw[black,thin] (6,0)--(6,3);
\draw[black,thin] (9,0)--(9,3);
\draw[black,thin] (12,0)--(12,3);
\draw[black,thin] (0,0)--(12,0);
\draw[black,thin] (0,3)--(12,3);
\draw (1,1.2) node [anchor=north] {\color{red} $2\xi$};
\draw (1.5,0) node [anchor=north] {$\xi$};
\draw (0,1.5) node [anchor=east] {$\mathbb{T}^3$};
\draw[red] (0,1) to[out=0,in=180] (3,2);
\draw[red] (3,1) to[out=0,in=180] (6,2);
\draw[red] (6,1) to[out=0,in=180] (9,2);
\draw[red] (9,1) to[out=0,in=180] (12,2);
\draw[red] (0,2) to[out=0,in=150] (1.4,1.57);
\draw[red] (1.6,1.43) to[out=-30,in=180] (3,1);
\draw[red] (3,2) to[out=0,in=150] (4.4,1.57);
\draw[red] (4.6,1.43) to[out=-30,in=180] (6,1);
\draw[red] (6,2) to[out=0,in=150] (7.4,1.57);
\draw[red] (7.6,1.43) to[out=-30,in=180] (9,1);
\draw[red] (9,2) to[out=0,in=150] (10.4,1.57);
\draw[red] (10.6,1.43) to[out=-30,in=180] (12,1);
\end{tikzpicture}
\caption{Four-fold cover of the embedding of the Klein bottle in the four-torus. Each black square represents a four-torus, where the horizontal direction is the $S^1$-factor representing the cycle $\xi\,,$ whereas the vertical direction is a three-torus. The red double helix is the image of $\gamma\,,$ which represents the cycle $2\xi\,.$}
\label{fig:Torus}
\end{center}
\end{figure}
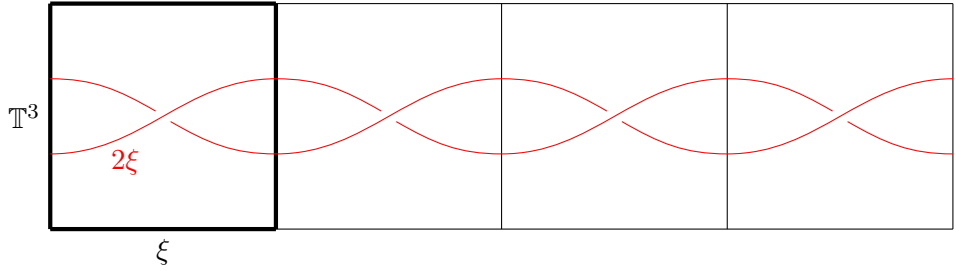
\end{example}
\section{Integration case IV: $D$ has normal crossings}\label{Normal crossing}
Now we treat the normal crossing case. Let $(M,D)$ be an elliptic pair, where $D$ has normal crossings, equipped with an elliptic ideal $\mathcal{I}_{|D|}$ such that for any $p\in M\,,$ there is a small neighbourhood $U\ni p$ and coordinates $(x_1,\dots,x_{n-2k},z_1,\dots,z_k)$ on $U$ such that $\mathcal{I}_{|D|}|_U\cong \langle \prod_iz_i\bz_i\rangle\,.$ 

We define the \emph{multiplicity} $\mult(p)$ of $p\in M$ to be the amount of hyperplanes that intersect normally at $p\,,$ i.e. $\mult(p):=\mathrm{codim}(\rho_{\mathcal{A}_{|D|}}[\mathcal{A}_{|D|,p}])/2\,.$ We then define the \emph{intersection locus of order} $k$ by $D[k]:=\{p\in M:\mult(p)=k\}\,,$ and we define $D(k)=\bigcup_{j\geq k}D[j]\,.$ $D[k]$ is a codimension $2k$ embedded submanifold (embeddedness follows from the normal crossing condition), but $D(k)$ is possibly singular, in particular, $D(1)=D\,.$

To any $p\in D[k]\,,$ we can associate a group $T_p\subseteq (\mathbb{Z}/2)^k\rtimes\Sigma_k\,,$ where $\Sigma_k$ is the symmetric group in $k$ variables, as follows: let $N$ be the connected component of $p$ in $D[k]\,.$ Then over $N\,,$ we have a $(\mathbb{Z}/2)^k\rtimes\Sigma_k$-principal bundle $P_N\to N\,$, where the fibre of $P_N$ at $p$ consists of all possible coorientations and orderings of the local hyperplanes intersecting at $p\,,$ with the obvious $(\mathbb{Z}/2)^k\rtimes\Sigma_k$-action. Pick any connected component $E$ of the total space of $P_N\,,$ then we define $T_p$ to be the subgroup of $(\mathbb{Z}/2)^k\rtimes\Sigma_k$ that fixes $E\,,$ turning $E$ into a $T_p$-principal bundle over $N\,.$ Note that $T_p$ is independent of the chosen connected component of $P_N\,,$ and that $T_p$ is, in fact, an invariant of $N\,,$ so we may also denote it by $T_N\,.$

\begin{definition}[Twist group]\label{twist group}
The \emph{twist group} of $N$ is the group $T_N$ constructed above. If $T_N=\{0\}\,,$ we say $N$ is \emph{untwisted and coorientable}. If $T_p=\{0\}$ for any $p\in M\,,$ we say $(M,D)$ is \emph{untwisted and coorientable}. 
\end{definition}
We will show that untwisted and coorientable elliptic divisors always admit a Hausdorff integration. Like before, the approach will be to construct the Lie groupoid around all the $k$-intersection loci and then glue them together. The integration around a $k$-intersection locus will be found by doing a sequence of blow-ups in a tubular neighbourhood around $D[k]\times D[k]$ in $M\times M\,$. Then the groupoid around the $D[k]$'s will be glued together to get the integration globally.

At the heart of the construction is the following lemma:
\begin{lemma}
Let $M_0$ denote the iterated blow-up of $\mathbb{R}^n\times\mathbb{C}^4\,,$ firstly with respect to the holomorphic ideal $\langle z_1,z_2\rangle$ and then with respect to $\beta^*\langle z_3,z_4\rangle\,,$ and let $M_1$ denote the iterated blow-up of $\mathbb{R}^n\times\mathbb{C}^4\,,$ firstly with respect to $\langle z_3,z_4\rangle$ and then with respect to $\beta^*\langle z_1,z_2\rangle\,.$ Then $M_0$ and $M_1$ are canonically diffeomorphic.
\end{lemma}
\begin{proof}
If $M$ is any manifold, then $M\times \mathbb{C}^m$ blown up along the holomorphic ideal $\langle z_1,\dots,z_m\rangle$ is canonically diffeomorphic to $M\times \widetilde{\mathbb{C}^m}\,,$ by the universal property of the blow-up. Thus we see that both iterated blow-ups are canonically diffeomorphic to $\mathbb{R}^n\times\widetilde{\mathbb{C}^2}\times\widetilde{\mathbb{C}^2}\,,$ which agree almost everywhere. Thus $M_0$ and $M_1$ are canonically diffeomorphic.
\end{proof}
Using this, we will first construct the elliptic groupoid around some untwisted and coorientable $N\in \pi_0(D[k])\,$. So let $U$ be a tubular neighbourhood of $N\,,$ such that $U$ has some trivialising charts $\{U_i\}\,,$ that are also normal crossing charts around $N\,.$ Since $N$ is untwisted and coorientable, $D\cap U$ consists of $k$ coorientable normal crossing divisors $H_1,\dots,H_k$ such that $H_1\cap\dots\cap H_k=N\,.$ We may define $\beta:\widetilde{U\times U}\to U\times U$ as the iterated blow-up of $H_i\times H_i\,,$ which does not depend on the order of doing the blow-ups by the Lemma above, so it's well defined and canonical. We define $H[j]=D[j]\cap U$ and $H(j)=D(j)\cap U$ and let $I_i$ denote the set of connected components of $H[i]\,.$ Using this, we define
$$\mathcal{G}_U:=\widetilde{U\times U}\setminus\bigcup_{i,j=0}^{k}\bigcup_{\substack{N_0\in I_i \\ N_1\in I_j \\ N_0\neq N_1}}\overline{
\beta^{-1}[N_0\times N_1]}\,,$$
together with the maps $s,t:\mathcal{G}_U\to U\,,$ defined by $s:=\pi_2\circ\beta$ and $t:=\pi_1\circ\beta\,.$ Then, we equip $U\setminus H(1)\times U\setminus H(1)$ with the pair groupoid structure. Like before, the pair groupoid structure extends to a Lie groupoid structure on $\mathcal{G}_U\,:$
\begin{lemma}
The pair groupoid structure on $U\setminus H(1)\times U\setminus H(1)$ extends uniquely to a groupoid structure on $\mathcal{G}_U\,.$ Moreover, this groupoid integrates the elliptic tangent bundle $\mathcal{A}_{|H(1)|}\to U$ of the pair $(U,H(1))\,.$
\end{lemma}
\begin{proof}
We shall work in coordinates on $U$ to give a local description of what $\mathcal{G}_U$ looks like. Let $N$ be a connected component of some $H[i]\,,$ and let $p,q\in N\,.$ We take normal crossing charts $\{U_p,(x_1,\dots,x_{n-2i},z_1,\dots,z_i)\}$ around $p$ and $\{U_q,(y_1,\dots,y_{n-2i},z_1',\dots,z_i')\}$ around $q\,.$ Since $N$ is untwisted and coorientable, we may assume the $z_j$ and $z_j'$ are ordered in a consistent way. Around $(p,q)\in U\times U\,,$ we get the blow-up chart
\begin{align*}
(x,y,a_1,\dots,a_i,b_1,\dots,b_i)\,,\quad
\beta(x,y,a_1,\dots,a_i,b_1,\dots,b_i)=(x,a_1,\dots,a_i,y,a_1b_1,\dots,a_ib_i)\,.
\end{align*}
Now, we note that $(x,y,a,b)\in U\setminus H(1)\times U\setminus H(1)$ corresponds precisely to the condition $a_j,b_j\neq 0$ for every $j=1,\dots,i\,.$ Moreover, $(N\times N)\cap (U_p\times U_q)$ corresponds to the locus $\{a_1,\dots,a_i=0\}\,,$ and, on $N\times N\,,$ the condition $b_{j_1},\dots,b_{j_l}=0$ corresponds precisely to $(x,y,0,b)\in \overline{\beta^{-1}[(N\cap U_p)\times (\{z'_{j_1},\dots,z'_{j_l}=0\}\setminus (N\cap U_q))]}\,.$ By repeating this argument for the other blow-up charts, we see that the blow-up chart given above in fact covers the entirety of $\beta^{-1}[U_p\times U_q]\cap \mathcal{G}_U\,,$ i.e. $\beta^{-1}[U_p\times U_q]\cap \mathcal{G}_U\cong\{(x,y,a,b):b_j\neq 0\text{ for }j=1,\dots,i\}\,.$ Using similar arguments as in Chapter \ref{Coorientable}, we then see that the groupoid structure extends uniquely to $\mathcal{G}\,,$ where, in these blown-up normal crossing coordinates on $N\,,$ we get the multiplication
$$m((x,y,0,b_1,\dots,b_i),(y,z,0,b'_1,\dots,b'_i))=(x,z,0,b_1b'_1,\dots,b_ib'_i)\,,$$
with inverses given by
$$\iota(x,y,0,b_1,\dots,b_i)=(y,x,0,b_1^{-1},\dots,b_i^{-1})$$
and units given by
$$u(x)=(x,x,0,1)\,,$$
for $x\in N\,.$

Lastly, by the same computation as in Chapter \ref{Coorientable}, this integrates the elliptic tangent bundle $\mathcal{A}_{|H(1)|}\,$.
\end{proof}
\begin{remark}\label{SmoothNormalCrossing}
Another way to see that this Lie groupoid integrates the elliptic tangent bundle is to note that $\mathcal{A}_{|H(1)|}\to U$ actually admits a fibre product decomposition $\mathcal{A}_{|H(1)|}\cong \mathcal{A}_{|H_1|}\times_{TM}\dots\times_{TM}\mathcal{A}_{|H_k|}\,,$ see \cite{Aldo}. We then know that the integration of $\mathcal{A}_{|H(1)|}$ can be described by the (strong-) fibre product $\mathcal{G}_{|H_1|}\times_{M\times M}\dots\times_{M\times M}\mathcal{G}_{|H_k|}\,,$ where $\mathcal{G}_{|H_i|}$ is the elliptic groupoid of $\mathcal{A}_{|H_i|}\,.$ A short computation gives that this groupoid agrees with the one described in the above proof.
\end{remark}

Using the construction above, assuming that the divisor is untwisted and coorientable around any $N\in \bigcup_{i}I_i\,,$ where $I_i:=\pi_0(D[i])\,,$ we can find an integration for the elliptic tangent bundle $\mathcal{A}_{|D|}$ on tubular neighbourhoods around every $N\in \bigcup_i I_i\,.$ If we then assume these tubular neighbourhoods are small enough, we have an integration $\mathcal{G}_{U_i}$ in some tubular neighbourhood $U_i$ around every $D[i]\,,$ where we take a disconnected union of the separate integrations, and we define $U_0:=M\setminus D$. What remains is to glue together these tubular neighbourhoods to get a global integration of the elliptic tangent bundle.

We note that for every $k<i\,,$ there is a canonical inclusion $\kappa_{ki}:(s_{U_i}^{-1}\cap t^{-1}_{U_i})[U_i\cap U_k]\to \mathcal{G}_{U_k}\,.$ This is easily seen from the local coordinate expressions given in the proof of the above Lemma. Moreover, we see that for $l<k<i\,,$ $\kappa_{li}=\kappa_{lk}\circ\kappa_{ki}\,.$ We then inductively define $\mathcal{G}$ by setting $\mathcal{G}_0:=(M\setminus D\times M\setminus D)_0:=\coprod_{N\in\pi_0(M\setminus D)}N\times N\,,$ and by defining $\mathcal{G}_i$ through the following pushout
\[\begin{tikzcd}
	{\mathcal{G}_{U_i}|_{U_i\cap U_{i-1}}} & {\mathcal{G}_{i-1}} \\
	{\mathcal{G}_{U_i}} & {\mathcal{G}_i\,,}
	\arrow[from=1-1, to=1-2]
	\arrow[from=1-1, to=2-1]
	\arrow[from=2-1, to=2-2]
	\arrow[from=1-2, to=2-2]
	\arrow["\ulcorner"{anchor=center, pos=0.125}, draw=none, from=1-1, to=2-2]
\end{tikzcd}\]
where the top map is induced by the $\kappa_{ik}\,.$ Note that this terminates after a finite amount of steps, as the normal crossing condition implies that the maximal number of divisors that intersect at some point $p\in M$ is $\lfloor \frac{n}{2}\rfloor\,,$ where $n$ is the dimension of $M\,.$
\begin{definition}[Elliptic groupoid IVa: Untwisted coorientable normal crossing version]\label{IVa}
The groupoid $\mathcal{G}$ constructed above is the \emph{elliptic groupoid} of the pair $(M,D)\,.$
\end{definition}
We note that $\mathcal{G}$ is naturally a Hausdorff Lie groupoid (see also \cite{GL12}), whose charts are described in the proof of the Lemma, and whose multiplication is induced by the groupoid structure on $\mathcal{G}_{U_i}\,$. A direct consequence of the above discussion is the following:
\begin{proposition}
The elliptic groupoid $\mathcal{G}$ of an untwisted, coorientable elliptic pair $(M,D)$ has the following properties:
\begin{enumerate} 
    \item The orbits of $\mathcal{G}$ are the connected components of $D[i]\,;$
    \item The isotropy of $\mathcal{G}$ over $D[i]$ is $(\mathbb{C}^*)^{i}\,.$
\end{enumerate}
\end{proposition}
Now we go to the general case where $(M,D)$ might be twisted and non-coorientable. The above construction needs to be modified in two ways: firstly, the local description of the Lie groupoid around a connected component of $D[i]$ needs to become bigger to deal with non-triviality of the twist groups. Secondly, the gluing maps need to be modified to deal with the fact that the groupoids will not match on overlaps after changing the local description.

\subsubsection*{Changing the local description}
To change the local description, we will follow the same procedure used to deal with the non-coorientable smooth case. Firstly, we note that around $x\in D[i]\,,$ we have normal crossing charts $(x_1,\dots,x_{n-2i},z_1,\dots,z_i)\,,$ where the $z_j$ are complex coordinates. The hypersurfaces intersecting at $x$ are the zero loci defined by $\{z_j=0\}$ for $j=1,\dots,i\,.$ Therefore, such a normal crossing chart induces both an order and a coorientation of the hypersurfaces intersecting at $x\,.$ If we cover a connected component $N$ of some $D[i]$ with such normal crossing charts, this defines a $(\mathbb{Z}/2)^i\rtimes\Sigma_i$-principal bundle over $N\,,$ where $\Sigma_i$ is the symmetric group in $i$ elements, that acts on $(\mathbb{Z}/2)^i$ by permutations. The trivialisations of this principal bundle arise from the normal crossing charts by identifying the induced coorientations and order of the hypersurfaces with the identity of $(\mathbb{Z}/2)^i\rtimes\Sigma_i\,,$ and all other coorientations and orders with the corresponding group element.

We can then extend this discrete principal bundle to a normal crossing tubular neighbourhood $U_N\to N$ to get a $(\mathbb{Z}/2)^i\rtimes \Sigma_i$-principal bundle that we shall denote by $\pi_N:P_N\to U_N\,.$ We note that $\pi_N^{-1}(U_N\cap D)$ then becomes an untwisted and coorientable normal crossing elliptic divisor inside $P_N\,.$ That means that we can construct the corresponding elliptic groupoid $\mathcal{G}_{P_N}\rightrightarrows P_N\,$, which we defined to be $s$-connected.

We note that the $(\mathbb{Z}/2)^i\rtimes\Sigma_i$ action lifts equivariantly to this elliptic groupoid, as follows: let $x\xleftarrow{g}y$ be an arrow in the elliptic groupoid of $(P_N,\pi_N^{-1}(U_N\cap D))\,,$ and pick normal crossing charts $U_x$ and $U_y$ around $N$ such that $x:=t(g)$ lies in $U_x$ and $y:=s(g)$ lies in $U_y\,.$ This also induces trivialisations of $P_N$ around $x$ and around $y\,,$ so we can describe $g$ in coordinates by 
$$(h_x,h_y,x_1,\dots,x_{n-2i},y_1,\dots,y_i,a_1,\dots,a_i,b_1,\dots,b_i)\,,$$ 
where $h_x,h_y\in(\mathbb{Z}/2)^i\rtimes \Sigma_i$ denote the group elements where $x$ and $y$ lie on in their respective trivialisations of $P_N\,,$ and $(x_1,\dots,x_{n-2i},y_1,\dots,y_{n-2i},a_1,\dots,a_i,b_1,\dots,b_i)$ are coordinates in $\widetilde{P_N\times P_N}$ induced by these normal crossing charts. Then in these coordinates, we define for any $h\in(\mathbb{Z}/2)^i\rtimes\Sigma_i\,,$
\begin{align*}
g\cdot h&=(h_x,h_y,x_1,\dots,x_{n-2i},y_1,\dots,y_i,a_1,\dots,a_i,b_1,\dots,b_i)\cdot h\\
&:=(h_xh,h_yh,x_1,\dots,x_{n-2i},y_1,\dots,y_i,h^{-1}\cdot(a_1,\dots,a_i),h^{-1}\cdot(b_1,\dots,b_i))\,,
\end{align*}
where $(\mathbb{Z}/2)^i\rtimes\Sigma_i$ acts on the $(a_1,\dots,a_i)$ and $(b_1,\dots,b_i)$ by permutation and then complex conjugation\footnote{The reason for the appearance of $h^{-1}$ is that $P_N$ is a right principal bundle, but this $(\mathbb{Z}/2)^i\rtimes\Sigma_i$-action is canonically on the left.}. Note that the blow-up coordinates around $(xh,yh)$ are also induced by the normal crossing chart on the base. We see that all structure maps of the elliptic groupoid of $P_N$ become equivariant under this action, therefore, we can quotient it out to get a groupoid $\mathcal{G}_N\rightrightarrows U_N\,.$ This will be the local model for the Lie groupoid around $N\,,$ which is Hausdorff.
\begin{definition}[Hausdorff local model]\label{HLM}
The groupoid $\mathcal{G}_N\rightrightarrows U_N$ is the \emph{Hausdorff local model} for the elliptic groupoid in a tubular neighbourhood of $N\,.$
\end{definition}

\subsubsection*{Describing the arrows}
We want to describe the space of arrows $\Hom_{\mathcal{G}_N}(x,y)$ between two points $x,y\in U_N\,.$ To do this, we note that this is only nonempty if $x,y\in S\in \pi_0((U_N\cap D)[j])\,,$ i.e. they lie in the same intersection locus in $U_N\,.$ We let $P_S$ denote the $(\mathbb{Z}/2)^j\rtimes\Sigma_j$-principal bundle over $S$ defined by pointwise coorienting and ordering the hypersurfaces intersecting at $S\,.$ Then we have two maps
\begin{align*}
&\Pi_1(S)\xrightarrow{\iota_*}\Pi_1(U_N)\xrightarrow{\Phi_N}(P_N\times P_N)/((\mathbb{Z}/2)^i\rtimes\Sigma_i)\\
&\Pi_1(S)\xrightarrow{\Phi_S}(P_S\times P_S)/((\mathbb{Z}/2)^j\rtimes\Sigma_j)\,,
\end{align*}
where $\Phi_N$ and $\Phi_S$ are groupoid morphisms describing the monodromy actions on the respective principal bundles, see also Subsection \ref{Monodromy}. We see that given normal crossing trivialisations of $U_N$ around $x$ and $y\,,$ we get an identification
\begin{equation}\label{arrowspace}
\Hom_{\mathcal{G}_N}(x,y)\cong (\mathbb{C}^*)^j\times (\Phi_N\circ\iota_*)(\Hom_{\Pi_1(S)}(x,y))\,.
\end{equation}

We also note that $\ker(\Phi_N\circ\iota_*)\subseteq\ker(\Phi_S)\,,$ which can be checked in normal crossing trivialisations of $U_N\,,$ essentially meaning that monodromy in $P_N$ contains strictly more information than monodromy in $P_S\,.$ In particular, $\Phi_S$ factors through the image of $\Phi_N\circ\iota_*\,.$ 
Composition of arrows in $\mathcal{G}_N$ then works as follows
$$(x_1,x_2,z,(\Phi_N\circ\iota_*)([\gamma]))\cdot(x_2,x_3,z',(\Phi_N\circ\iota_*)([\gamma']))=(x_1,x_3,z\cdot \Phi_S([\gamma])(z'),(\Phi_N\circ\iota_*)([\gamma]\cdot[\gamma']))\,,$$
where $z,z'\in(\mathbb{C}^*)^j$ and $\Phi_S([\gamma]):(P_S)_{x_2}\to (P_S)_{x_1}\,,$ in any trivialisations around $x_1$ and $x_2\,,$ defines an element of $(\mathbb{Z}/2)^j\rtimes\Sigma_j$ acting on the left, which we let act on the $z'\,,$ where we again note that all of this is equivariant under changes in normal crossing trivialisations.

In particular, the isotropy at the point $x$ is $(\mathbb{C}^*)^j\rtimes(\varphi_N\circ\iota_*)(\pi_1(S,x))\,,$ where $(\varphi_N\circ\iota_*)(\pi_1(S,x))$ acts through its projection onto $(\mathbb{Z}/2)^j\rtimes\Sigma_j\,,$ whose image is actually $T_S$ that we defined in Definition \ref{twist group}.

\subsubsection*{Finding the gluing maps}
We then turn to defining the gluing maps. The idea is as follows: assume that we found the groupoid $\mathcal{G}_{i-1}$ over $M\setminus D(i)\,,$ where $D(i)=\bigcup_{k\geq i}D[k]\,,$ then to take a pushout to define $\mathcal{G}_i\,$, we need to specify a map $\mathcal{G}_N|_{U_N\setminus N}\to \mathcal{G}_{i-1}$ for every $N\in \pi_0(D[i])\,.$ We shall give such a map using a certain open cover of the Hausdorff local model $\mathcal{G}_N$ (to be specified below), and defining the required map on that open cover. 

Firstly, note that for $S\in\pi_0((U_N\cap D)[j])\,,$ we naturally have two groupoids defined around it, $\mathcal{G}_N|_S$ and $\mathcal{G}_{N_S}|_S\,,$ where $N_S\in\pi_0(D[j])$ is the intersection locus in $M$ that contains $S\,.$ Moreover, there is a map $\kappa_{N_S,N}:\mathcal{G}_N|_S\to \mathcal{G}_{N_S}|_S\,,$ that acts on $\Hom_{\mathcal{G}_N(x,y)}\cong (\mathbb{C}^*)^j\times(\Phi_N\circ\iota_*)(\Hom_{\Pi_1(S)}(x,y))$ as
$$\kappa_{N_S,N}(z,(\Phi_N\circ\iota_*)([\gamma]))=(z,\Phi_S([\gamma]))\,,$$
using the identification in Equation \ref{arrowspace}. In general, this map is neither injective nor surjective.

We can extend this map to a small open around $\mathcal{G}_N|_S$ as follows: we have an inclusion $\iota_S:U_N\cap (U_{N_S}|_S)\to U_{N_S}|_S\,,$ which then defines a map of groupoids $(\iota_S)_*:\Pi_1(U_N\cap (U_{N_S}|_S))\to \Pi_1(U_{N_S}|_S)\,.$ We see that $\ker((\iota_S)_*)$ consists of loops in $U_N\cap (U_{N_S}|_S)$ that are contractible inside $U_{N_S}|_S\,.$ Such loops never have monodromy, thus $\ker((\iota_S)_*)\subseteq \ker(\Phi_N\circ(\iota_N)_*)\,,$ where $\iota_N:U_N\cap (U_{N_S}|_S)\hookrightarrow U_N\,.$ In particular, $\Phi_N\circ(\iota_N)_*$ factors through the image of $(\iota_S)_*\,.$

Moreover, if we let $\Phi_{N_S}:\Pi_1(U_{N_S}|_S)\to ((P_{N_S}\times P_{N_S})/((\mathbb{Z}/2)^j\rtimes\Sigma_j))$ denote the monodromy inside $N_S\,,$ we see, like before, that 
$$\ker{(\Phi_N\circ(\iota_N)_*)|_{(\iota_S)_*(\Pi_1(U_N\cap U_{N_S}|_S))}}\subseteq \ker(\Phi_{N_S})\,.$$
In particular, $\Phi_{N_S}\circ(\iota_S)_*:\Pi_1(U_N\cap U_{N_S}|_{S})\to (P_{N_S}\times P_{N_S})/((\mathbb{Z}/2)^j\rtimes\Sigma_j))$ factors through the image of $\Phi_N\circ(\iota_N)_*\,,$ which is precisely $(P_N\times P_N)/((\mathbb{Z}/2)^i\rtimes\Sigma_i)|_{U_N\cap U_{N_S}|_{S},0}\,,$ where the subscript $0$ indicates we're restricting to the $s$-component of the units. In total, we get a map
$$(P_N\times P_N)/((\mathbb{Z}/2)^i\rtimes\Sigma_i)|_{U_N\cap U_{N_S}|_{S},0}\to (P_{N_S}\times P_{N_S})/((\mathbb{Z}/2)^j\rtimes\Sigma_j))\,,$$
which extends to a map $\kappa_{N_S,N}:\mathcal{G}_N|_{U_N\cap U_{N_S}|_S,0}\to \mathcal{G}_{N_S}$ as follows: let $x,y\in S'\in\pi_0((U_N\cap U_{N_S}|_S\cap D)[k])$ for some $k\leq j\,.$ Then
$$\Hom_{\mathcal{G}_N|_{U_N\cap U_{N_S}|_S,0}}(x,y)\cong (\mathbb{C}^*)^k\times (\Phi_{N}\circ(\iota_N)_*)(\Hom_{\Pi_1(U_N\cap U_{N_S}|_S)}(x,y))\,,$$
such that 
\begin{equation}\label{Kappa}
\kappa_{N_S,N}(z,(\Phi_N\circ(\iota_N)_*)([\gamma]))=(z,(\Phi_{N_S}\circ(\iota_{S})_*)([\gamma]))\,.
\end{equation}

For every $N_1\in \pi_0(D[i])\,,$ $N_2\in \pi_0(D[j])$ and $N_3\in \pi_0(D[k])\,,$ for $k\leq j\leq i\,,$ we get maps 
\begin{align*}
&\kappa_{N_2,N_1}:\mathcal{G}_{N_1}|_{U_{N_1}\cap (U_{N_2}|_{N_2\cap U_{N_1}}),0}\to \mathcal{G}_{N_2}\\
&\kappa_{N_3,N_2}:\mathcal{G}_{N_2}|_{U_{N_2}\cap (U_{N_3}|_{N_3\cap U_{N_2}}),0}\to \mathcal{G}_{N_3}\,.
\end{align*}
Moreover, it follows from Equation \ref{Kappa} that 
\begin{equation}\label{nicekappa}
\kappa_{N_3,N_2}\circ\kappa_{N_2,N_1}=\kappa_{N_3,N_1}\,,
\end{equation}
under appropriate restrictions. 

Now, we see that the $\mathcal{G}_N|_{U_N\cap (U_{N'}|_{N'\cap U_N}),0}$ cover $\mathcal{G}_N|_{U_N\setminus N}\,.$ Therefore, we can inductively define $\mathcal{G}_0:=(M\setminus D\times M\setminus D)_0\,,$ and $\mathcal{G}_i$ through the pushout (of topological spaces)
\[\begin{tikzcd}
	{\bigcup_{N\in\pi_0(D[i])}\mathcal{G}_N|_{U_N\setminus N}} & {\mathcal{G}_{i-1}} \\
	{\bigcup_{N\in\pi_0(D[i])}\mathcal{G}_N} & {\mathcal{G}_{i}\,,}
	\arrow["{\kappa_i}", from=1-1, to=1-2]
	\arrow[hook', from=1-1, to=2-1]
	\arrow[from=1-2, to=2-2]
	\arrow[from=2-1, to=2-2]
        \arrow["\ulcorner"{anchor=center, pos=0.125}, draw=none, from=1-1, to=2-2]
\end{tikzcd}\]
where $\kappa_i$ is the continuous map induced by the $\kappa_{N',N}$ maps described above, which is well-defined because it is defined on an open cover of $\bigcup\mathcal{G}_N|_{U_N\setminus N}\,,$ such that the maps agree on overlaps, by Equation \ref{nicekappa}.

We note again that this terminates after a finite amount of steps, so we get a well defined (a priori set-theoretical) groupoid $\mathcal{G}\rightrightarrows M\,.$ The groupoid structure on $\mathcal{G}$ over any orbit $N\in\pi_0(D[i])$ is induced by the natural identification $\mathcal{G}|_N\cong\mathcal{G}_N|_N\,.$ Then we just have to show that $\mathcal{G}$ is a (not-necessarily-Hausdorff) smooth manifold and that the groupoid maps satisfy the Lie groupoid axioms, which we will do by giving an atlas of $\mathcal{G}\,.$

\subsubsection*{Describing the local charts}
Like before, $\mathcal{G}|_{M\setminus D}$ is just a disjoint union of pair groupoids, so it is naturally a Lie groupoid. Moreover, given any $x\xleftarrow{g} y\in \mathcal{G}$ for $x,y\in N\in\pi_0(D[i])\,,$ together with normal crossing charts $U_x$ and $U_y$ centred at $x$ and $y$ respectively, we have a coordinate chart of $\mathcal{G}$ containing $g\,,$ that is induced by the blow-up chart $(x_1,\dots,x_{n-2i},y_1,\dots,y_{n-2i},a_1,\dots,a_i,b_1,\dots,b_i)\,.$ These charts clearly are smoothly compatible with each other, and in these charts it is clear that all groupoid maps satisfy the Lie groupoid axioms.

\begin{definition}[Elliptic groupoid IVb: Normal crossing version]
The groupoid $\mathcal{G}$ constructed above is the \emph{elliptic groupoid} of the elliptic pair $(M,D)\,.$
\end{definition}
Note that this groupoid agrees with the one given in Definition \ref{IVa} in the case where $D$ is untwisted and coorientable.

We have the following:
\begin{proposition}
The Lie groupoid $\mathcal{G}$ obtained by this construction satisfies
\begin{enumerate} 
    \item The orbits of $\mathcal{G}$ are the connected components of every $D[i]\,;$
    \item The isotropy of $\mathcal{G}$ over some $N\in\pi_0(D[i])$ is $(\mathbb{C}^*)^i\rtimes T_N\,,$ where $T_N$ is the twist group of $N\,$.
    \item $\mathcal{G}$ is Hausdorff if and only if $(M,D)$ is untwisted and coorientable.
\end{enumerate}
\end{proposition}
The following is a consequence of Proposition 3.8 in \cite{AS07}:
\begin{corollary}
Let $D\subseteq M$ be an elliptic divisor with normal crossings. Then the elliptic groupoid $\mathcal{G}$ of $\mathcal{A}_{|D|}\to M$ is final in the category of $s$-connected integrations of $\mathcal{A}_{|D|}\,.$ 
\end{corollary}
Analogously to Chapter \ref{Not coorientable}, we will now study existence of Hausdorff integrations. In the smooth case, we saw that obstructions to Hausdorff integrations come from coorientation-reversing loops on $D$ that are nullhomotopic in $M\,.$ By complete analogy, one could expect that a similar thing is true in this setting: the only obstructions to the existence of Hausdorff integrations come from loops on $D[i]$ that twist the divisor and are nullhomotopic in $M\,.$ To prove this, we will first show that the Hausdorff local model from Definition \ref{HLM} is in fact the final $s$-connected Hausdorff integration. That is, analogously to Proposition \ref{FinalHausdorff}, we have
\begin{proposition}\label{FinalHausdorffII}
Let $D$ be a normal crossing divisor, let $N\in\pi_0(D[i])$ and let $U$ be a normal crossing tubular neighbourhood of $N\,.$ Then $\mathcal{G}_N\rightrightarrows U$ is final in the category of $s$-connected Hausdorff integrations of $\mathcal{A}_{|H(1)|}\to U\,.$
\end{proposition}
\begin{proof}
The proof is similar to that of Proposition \ref{FinalHausdorff}, so we will be brief. First, note that any discrete normal subgroup of $(\mathbb{C}^*)^i\rtimes((\mathbb{Z}/2)^i\rtimes\Sigma_i)$ is contained in $(\mathbb{C}^*)^i\times\{1\}\,.$  Then, denote by $p:\mathcal{G}_{U}^\mathrm{ssc}\to\mathcal{G}_N$ the unique groupoid map, and by $\mathcal{G}_{\mathcal{N}'}\cong\mathcal{G}_U^\mathrm{ssc}/\mathcal{N}'$ any Hausdorff integration of $\mathcal{A}_{|H(1)|}\,,$ written as a groupoid quotient of $\mathcal{G}_U^\mathrm{ssc}$ by a closed, discrete, normal, totally disconnected subgroupoid $\mathcal{N}'\,.$ Then, by closedness of $\mathcal{N}'\,,$ we see that if $p(\mathcal{N}')|_{U\setminus H(1)}$ contains an arrow $g\notin u_{\mathcal{G}_N}(U\setminus H(1))\,,$ then for $t\in(0,1]\,,$ no continuous lift of $tg$ to $\mathcal{N}'$ may converge in $\mathcal{G}_U^\mathrm{ssc}$ as $t\to 0\,.$ However, by the description of $\mathcal{G}_U^\mathrm{ssc}$ as $\mathcal{A}_{|H(1)|}$-paths up to $\mathcal{A}_{|H(1)|}$-homotopy, such a lift must converge. Thus, we see that $p(\mathcal{N}')$ is a subgroupoid such that $p(\mathcal{N}')|_{U\setminus H(1)}=u_{\mathcal{G}_N}(U\setminus H(1))\,,$ i.e. $p(\mathcal{N}')=u_{\mathcal{G}_N}(U)\,,$ completing the proof.
\end{proof}

And as in Theorem \ref{ExistenceOfHausdorffIntegration}, we get
\begin{theorem}\label{ExistenceOfHausdorffIntegrationII}
The elliptic tangent bundle $\mathcal{A}_{|D|}\to M$ of a normal crossing elliptic pair $(M,D)$ has a Hausdorff integration if and only if for every $i\geq 1$ and any $N\in \pi_0(D[i])\,,$ $\ker((\iota_N)*)\subseteq \ker(\varphi_N)\,,$ where $\iota_N:N\hookrightarrow M$ is the inclusion, and $\varphi_N:\pi_1(N)\to T_N$ is the monodromy representation.
\end{theorem}
\begin{proof}
If $\mathcal{G}\rightrightarrows M$ is any Hausdorff integration of $\mathcal{A}_{|D|}\,,$ then for any $i\geq 1$ and any $N\in\pi_0(D[i])\,,$ together with a normal crossing tubular neighbourhood $U\to N\,,$ we have that $\mathcal{G}|_{U,0}$ is a Hausdorff integration of $\mathcal{A}_{|U\cap D|}\,,$ thus giving us a map $\mathcal{G}|_{U,0}\to\mathcal{G}_N\,.$ Composing this with the unique map $\mathcal{G}^\mathrm{ssc}\to\mathcal{G}\,,$ we obtain a map $\mathcal{G}^\mathrm{ssc}|_{U,0}\to\mathcal{G}_N\,,$ i.e. the unique map $\mathcal{G}_U^\mathrm{ssc}\to\mathcal{G}_N$ factors through $\mathcal{G}^\mathrm{ssc}|_{U,0}\,.$ At the level of isotropies, this gives a map $(\iota_N)_*(\pi_1(U\setminus (U\cap D)))\to T_N\,,$ which factors the map $\psi:\pi_1(U\setminus (U\cap D)))\to T_N\,.$ From the discussion in Subsection \ref{Monodromy}, we know that $\psi([\gamma])=\varphi_N(\gamma)\,.$ In particular, since any loop on $N$ is homotopic to a loop in $U\cap(U\setminus D)\,,$ we see that the monodromy map $\varphi_N:\pi_1(N)\to T_N$ factors through a map $(\iota_N)_*(\pi_1(N))\to T_N\,,$ in particular, $\ker((\iota_N)_*)\subseteq \ker(\varphi_N)\,.$

Conversely, suppose that for every $i\geq 1$ and any $N\in \pi_0(D[i])\,,$ $\ker((\iota_N)*)\subseteq \ker(\varphi_N)\,.$ Then we can take the universal cover of $M\,,$ such that $(\widetilde{M},\widetilde{D})$ is now untwisted and coorientable, since any loop on any component of $\widetilde{D}[i]$ must necessarily come from a loop in $\ker((\iota_N)*)$ for some $N\,,$ which has no monodromy by assumption. Taking a quotient of the associated elliptic groupoid by $\pi_1(M)$ then gives a Hausdorff integration of $\mathcal{A}_{|D|}\,.$
\end{proof}

We note that if there are no self-crossings of divisors, i.e. if $D$ is a union smooth divisors that cross each other normally, then there's a more elementary description of the elliptic groupoid that doesn't use these pushouts. Analogous to Remark \ref{SmoothNormalCrossing}, we have the following:

\begin{proposition}
Suppose $D,D'\subseteq M$ are two normal crossing elliptic divisors such that $D+D'$ is a normal crossing divisor. Let $\mathcal{G}_D$ and $\mathcal{G}_{D'}$ be Lie groupoids integrating the elliptic tangent bundle of $(M,D)$ and $(M,D')\,,$ respectively. Then the (strong) fibre product $\mathcal{G}_D\times_{M\times M}\mathcal{G}_{D'}$ is a Lie groupoid integrating the elliptic tangent bundle of $(M,D+D')\,.$ Moreover, if $\mathcal{G}_D$ and $\mathcal{G}_{D'}$ are both Hausdorff, then so is $\mathcal{G}_D\times_{M\times M}\mathcal{G}_{D'}\,.$
\end{proposition}
\begin{proof}
Since $\mathcal{G}_D\to M\times M$ and $\mathcal{G}_{D'}\to M\times M$ are open embeddings followed by a blow-down map, the cokernel of their differentials are directions normal to $D\times D$ or $D'\times D'\,,$ respectively. So because $D$ and $D'$ intersect normally, we see that $\mathcal{G}_D\to M\times M$ and $\mathcal{G}_{D'}\to M\times M$ are transverse. Thus $\mathcal{G}_{D}\times_{M\times M}\mathcal{G}_{D'}$ is a smooth manifold, hence defines a Lie groupoid. Moreover, if $\mathcal{G}_D$ and $\mathcal{G}_{D'}$ are Hausdorff, their strong fibre product is Hausdorff too.

That the associated Lie algebroid is indeed the elliptic tangent bundle of $(M,D+D')$ follows locally from Remark \ref{SmoothNormalCrossing}.
\end{proof}

\section{Elliptic divisors on elliptic groupoids}\label{EllipticIdeal}
In this chapter, let $D$ be a normal crossing elliptic divisor with elliptic ideal $\mathcal{I}_{|D|}\,.$ We will show that the associated elliptic groupoid $\mathcal{G}$ together with the divisor $\mathcal{D}:=s^{-1}[D]=t^{-1}[D]$ naturally admits the structure of an elliptic pair compatible with the elliptic structure present on $M\,$. We will prove the following:

\begin{proposition}\label{ellipticideal}
The ideal sheaf $\mathcal{I}_{|\mathcal{D}|}= s^*\mathcal{I}_{|D|}= t^*\mathcal{I}_{|D|}\,$, is an elliptic ideal sheaf with divisor $\mathcal{D}\,,$ such that $s,t:(\mathcal{G},\mathcal{D})\to (M,D)$ are maps of elliptic pairs, i.e. strong maps of pairs that pull back $\mathcal{I}_{|D|}$ to $\mathcal{I}_{|\mathcal{D}|}\,$.
\end{proposition}
\begin{proof}
Note that this is a local statement, so we can pick an $N\in D[i]$ and proceed in a normal crossing tubular neighbourhood $U_N$ of $N\,.$ Since $\mathcal{G}|_{U_N}$ was defined through taking an untwisted coorientable cover $\widetilde{U_N}\to U_N\,,$ integrating to a groupoid, and then taking a quotient by a discrete group action, local coordinates around a connected component of $\Hom_{\mathcal{G}}(p,q)\,,$ for some $p,q\in N\,,$ are obtained by the local coordinates around some lifts of $p$ and $q$ to $\widetilde{U_N}\,,$ so for describing the elliptic ideal, we may assume $U_N$ is untwisted and coorientable. 

Since $U_N$ is an untwisted coorientable normal crossing tubular neighbourhood, $D\cap U_N\cong \bigcup_{j=1}^i H_j\,,$ where $H_j$ are smooth divisors with elliptic ideals $\mathcal{I}_{|H_j|}$ satisfying $\prod_j \mathcal{I}_{|H_j|}=\mathcal{I}_{|D|}\,.$ Pick a subset $J\subseteq \{1,\dots,i\}$ and pick points 
$$p,q\in\left(\bigcap_{j\in J}H_j\right)\setminus \bigcup_{j\in J^c}H_j\,,$$ 
where $J^c$ denotes the complement of $J$ in $\{1,\dots,i\}\,.$ That is, pick $p$ and $q$ that lie in the intersection locus of $\{H_j\}_{j\in J}\,,$ but not on any of the other $H_j\,$. Without loss of generality, we may then assume $J=\{1,\dots,|J|\}\,.$ Then around $\Hom_{\mathcal{G}_{U_N}}(p,q)\,,$ we have coordinates 
$$\{(x,y,a,b)|x,y\in\mathbb{R}^{n-2|J|}\,,\,a\in\mathbb{C}^{|J|}\,,\,b\in(\mathbb{C}^*)^{|J|}\}\,,$$ 
such that $s(x,y,a,b)=(y,ab)$ and $t(x,y,a,b)=(x,a)\,,$ where $ab=(a_1b_1,\dots,a_{|J|}b_{|J|})\,,$ and $s^*\mathcal{I}_{|H_j|}=\langle|a_jb_j|^2\rangle\,,$ $ t^*\mathcal{I}_{|H_j|}=\langle|a_j|^2\rangle\,.$ Since $b_j$ is an invertible function, we see that both these ideals agree. Therefore, we can put $\mathcal{I}_{|\mathcal{D}|}=s^*\mathcal{I}_{|D|}= t^*\mathcal{I}_{|D|}\,,$ and the computation above shows this is indeed a well defined elliptic ideal. That $s$ and $t$ are then morphisms of elliptic pairs is by definition. 
\end{proof}

\section{Symplectic integrations of elliptic symplectic structures}\label{SympGroupoid}
As a concluding chapter, we will describe a local model for a different groupoid: the symplectic groupoid associated to an elliptic Poisson structure. We recall the following from \cite{CG15}:
\begin{definition}[Elliptic symplectic form]
Let $D\subseteq M$ be a smooth elliptic divisor, then an \emph{elliptic symplectic form} on $(M,D)$ is a nondegenerate $d_{\mathcal{A}_{|D|}}$-closed two-form $\omega\in\Omega^2(\mathcal{A}_{|D|}):=\Gamma(\Lambda^2(\mathcal{A}_{|D|})^*)\,.$
\end{definition}
Nondegeneracy means that $\omega^\flat:\mathcal{A}_{|D|}\to \mathcal{A}_{|D|}^*$ is an isomorphism, hence we can invert it to get $(\omega^\flat)^{-1}:\mathcal{A}_{|D|}^*\to \mathcal{A}_{|D|}\,.$ Then we define $\pi^\sharp:=\rho\circ(\omega^\flat)^{-1}\circ\rho^*:T^*M\to TM\,,$ where $\rho:\mathcal{A}_{|D|}\to TM$ is the anchor map. Since $\rho$ is an isomorphism away form $D\,,$ $\omega$ is an honest symplectic two-form on $M\setminus D\,,$ so we see that the bivector $\pi\in \mathfrak{X}^2(M)$ associated to $\pi^\sharp$ satisfies $[\pi,\pi]=0$ away from $D\,,$ so by closedness, $\pi$ is a Poisson bivector.
\begin{definition}[Elliptic Poisson structure]\label{EllPois}
The bivector associated to an elliptic symplectic form is an \emph{elliptic Poisson structure} on $(M,D)$.
\end{definition}
Using this Poisson bivector, we get a cotangent Lie algebroid associated to any elliptic symplectic structure, which we denote by $\mathcal{A}:=(T^*M,\pi^\sharp,[-,-]_\pi)\,.$ Since $\pi^\sharp$ is an isomorphism on $M\setminus D\,,$ we know from \cite{Debord} and \cite{CF03} that this Lie algebroid is integrable. There are two local models we will consider: firstly, an elliptic Poisson structure with \emph{elliptic residue}, that is, the elliptic symplectic structure restricts to an honest symplectic structure on $D\,,$ and the case of an elliptic Poisson structure with zero elliptic residue, in which case $\pi|_D^\sharp$ has a two dimensional kernel. 

We shall treat these two cases only locally. We leave the question of figuring out the global topology of the integration of an elliptic Poisson structure on $(M,D)$ for future work.

\subsection{Nonzero elliptic residue}
For the elliptic Poisson structure with elliptic residue, we use the normal form where $M=\mathbb{R}^2\,,$ $D=\{0\}$ and the elliptic ideal is generated by $r^2\,,$ in which case $\mathcal{A}\cong\mathrm{span}(r^2\partial_x,r^2\partial_y)\,,$ as $\omega=f d\log r\wedge d\theta\,,$ so $\pi=\tfrac{1}{f} r\partial_r\wedge \partial_\theta=\tfrac{r^2}{f}\partial_x\wedge\partial_y$ for some nonvanishing $f\in C^\infty(\mathbb{R}^2)\,.$

The idea for the integration of this local model is based on the integration of the $b$-tangent bundle and the elliptic tangent bundle. We will proceed by doing some form of weighted blowup that relies on the following observation:
\begin{lemma}\label{Blowup}
For every $(x_1,x_2,y_1,y_2)\in(\mathbb{R}^2\setminus\{0\})\times(\mathbb{R}^2\setminus\{0\})\,,$ there is a unique pair $(a,b)\in\mathbb{R}^2$ such that
\begin{equation}\label{BlowupEqn}
\begin{cases}
y_1=a(x_1^2+x_2^2)+x_1\,;\\
y_2=b(x_1^2+x_2^2)+x_2\,.
\end{cases}
\end{equation}
\end{lemma}
\begin{proof}
By immediate computation, we see
\begin{equation}
\begin{cases}\label{InverseBlowup}
a=\frac{y_1-x_1}{x_1^2+x_2^2}\,;\\
b=\frac{y_2-x_2}{x_1^2+x_2^2}\,.
\end{cases}
\end{equation}
\end{proof}
Intuitively, the kind of weighted blowup that we will do will be to replace the origin in $\mathbb{R}^4$ by a copy of $\mathbb{R}^2\,,$ such that any curve going to $(0,0)$ along the surface cut out by Equations \eqref{BlowupEqn}, will converge to the point $(a,b)\,.$ More precisely, define 
$$\mathcal{G}:=\mathbb{R}^4\setminus \left\{\left(x_1,x_2,\frac{-x_1}{x_1^2+x_2^2},\frac{-x_2}{x_1^2+x_2^2}\right):(x_1,x_2)\neq 0\right\}\,,$$
and define $\varphi:\mathcal{G}\setminus(\{0\}\times\mathbb{R}^2)\to (\mathbb{R}^2\setminus\{0\})\times(\mathbb{R}^2\setminus\{0\})$ by
$$\varphi(x_1,x_2,a,b):=(x_1,x_2,a(x_1^2+x_2^2)+x_1,b(x_1^2+x_2^2)+x_2)\,.$$
We see that $\varphi$ is a diffeomorphism with inverse $\varphi^{-1}(x_1,x_2,y_1,y_2)=(x_1,x_2,a,b)\,,$ with $(a,b)$ as in Equations \eqref{InverseBlowup}. We have the following:

\begin{theorem}
The pair groupoid structure on $(\mathbb{R}^2\setminus\{0\})\times(\mathbb{R}^2\setminus\{0\})\,,$ pulled back to $\mathcal{G}\setminus(\{0\}\times\mathbb{R}^2)$ along $\varphi\,,$ extends uniquely to a Lie groupoid structure $\mathcal{G}\rightrightarrows\mathbb{R}^2\,,$ which integrates $\mathcal{A}\,,$ the cotangent Lie algebroid of an elliptic Poisson structure on $(\mathbb{R}^2,\{0\})\,.$
\end{theorem}
\begin{proof}
Note that if the extension exists, it is necessarily unique, as $\im(\varphi^{-1})$ is dense in $\mathcal{G}\,,$ moreover, if all maps extend smoothly, the same argument shows that all groupoid relations are still satisfied, so we only need to show that all maps extend smoothly. We see that $s:=\pi_2\circ\varphi$ and $t:=\pi_1\circ\varphi$ are given by
\begin{align*}
s(x_1,x_2,a,b)&=(a(x_1^2+x_2^2)+x_1,b(x_1^2+x_2^2)+x_2)\,;\\
t(x_1,x_2,a,b)&=(x_1,x_2)\,.
\end{align*}
These extend to $\mathcal{G}$ by $s(0,0,a,b)=0$ and $t(0,0,a,b)=0\,.$ We also see $\iota:=\varphi^{-1}\circ\iota\circ\varphi$ is given by 
$$\iota(x_1,x_2,a,b)=(a(x_1^2+x_2^2)+x_1,b(x_1^2+x_2^2)+x_2,-a,-b)\,,$$
which extends smoothly to $\mathcal{G}$ by $\iota(0,0,a,b)=(0,0,-a,-b)\,.$ We have that $u:=\varphi^{-1}\circ\Delta$ is given by
$$u(x_1,x_2)=(x_1,x_2,0,0)\,,$$
which extends smoothly to $\mathcal{G}$ by $u(0,0)=0\,.$ Finally, multiplication $m:=\varphi^{-1}\circ m\circ(\varphi,\varphi)$ is given by
\begin{align*}
m((x_1,x_2,a_1,b_1),&(a_1(x_1^2+x_2^2)+x_1,b_1(x_1^2+x_2^2)+x_2,a_2,b_2))\\
&=(x_1,x_2,a_1+a_2+O(x_1,x_2),b_1+b_2+O(x_1,x_2))\,,
\end{align*}
where $O(x_1,x_2)$ are terms polynomial in $(x_1,x_2)$ of degree 1 or higher. We see that $m$ extends smoothly to $\mathcal{G}\times_M\mathcal{G}$ by $m((0,0,a_1,b_1),(0,0,a_2,b_2))=(0,0,a_1+a_2,b_1+b_2)\,.$ Thus, $(\mathcal{G},s,t,m,\iota,u)$ becomes a Lie groupoid. To show that this indeed integrates $\mathcal{A}\,,$ we compute
$$ds|_{(x_1,x_2,a,b)}=\begin{pmatrix}
2ax_1+1&2ax_2&(x_1^2+x_2^2)&0\\
2bx_1&2bx_2+1&0&(x_1^2+x_2^2)
\end{pmatrix}\,.$$
Thus,
$$\ker(ds)|_u=\mathrm{span}(((x_1^2+x_2^2)\partial_{x_1}+\partial_a),((x_1^2+x_2^2)\partial_{x_2}+\partial_b))\,.$$
Applying 
$$dt=\begin{pmatrix}
1&0&0&0\\
0&1&0&0
\end{pmatrix}\,,$$
we see that $\mathrm{Lie}(\mathcal{G})=\mathrm{span}(r^2\partial_{x_1},r^2\partial_{x_2})\,,$ which we know to be isomorphic to $\mathcal{A}$ by the discussion below Definition \ref{EllPois}. This completes the proof.
\end{proof}
We have the following:
\begin{proposition}
$\mathcal{G}$ has the following properties:
\begin{enumerate}
    \item The orbits of $\mathcal{G}$ are $\mathbb{R}^2\setminus\{0\}$ and $\{0\}\,$;
    \item The isotropy of $\mathcal{G}$ over $\mathbb{R}^2\setminus\{0\}$ is trivial;
    \item The isotropy of $\mathcal{G}$ over $\{0\}$ is $\mathbb{R}^2\,$.
\end{enumerate}
\end{proposition}
\begin{remark}
The above can be generalised to find the integration of any Lie algebroid on $\mathbb{R}^n$ of the form $\mathcal{A}=\mathrm{span}\{r^{2m_i}\partial_i\}_{i=1,\dots,n}\,,$ where $m_i\geq 1$ are integers. In the particular case where $n=1\,,$ the exponent can also be odd, i.e. the same procedure also applies to find a local model for the integration of $b^k$-tangent bundles in the sense of \cite{BPW23}.
\end{remark}

Since this groupoid again agrees with the pair groupoid over a dense subset of $M\,,$ finding the multiplicative symplectic structure on $\mathcal{G}$ that is associated to the elliptic Poisson structure is straightforward, as the multiplicative symplectic structure on the pair groupoid of a symplectic manifold $(M,\omega)$ is given by just $t^*\omega-s^*\omega\,,$ see e.g. \cite{CFM21}. Thus we have the following
\begin{corollary}\label{MultSymp1}
The multiplicative symplectic structure on $\mathcal{G}$ associated to $\omega=d\log r\wedge d\theta$ is given by
\begin{align*}
\Omega &= \frac{1}{(a^2+b^2)(x_1^2+x_2^2)+2ax_1+2bx_2+1}((a^2+b^2)(x_1^2+x_2^2)dx_1\wedge dx_2\\
&\quad-2bx_1dx_1\wedge da+(2ax_1+1)dx_1\wedge db-(2bx_2+1)dx_2\wedge da+2ax_2dx_2\wedge db)\,.
\end{align*}
\end{corollary}
\begin{remark}
The $s$-simply connected cover of this symplectic groupoid has already been found by Cuesta, Dazord and Hector in \cite{ADH89}, where they describe the $s$-simply connected symplectic integration of the Poisson structure $(x^2+y^2)\partial_x\wedge\partial_y$ on $\mathbb{R}^2\,.$ The integration is an $s$-simply connected Lie groupoid $\Gamma\rightrightarrows \mathbb{R}^2\,$, whose manifold of arrows is $\mathbb{C}\times\mathbb{C}\,.$ The unique groupoid map $\psi:\Gamma\to \mathcal{G}$ is given by (following the conventions in \cite{ADH89}): 
$$\psi(Z,z)=\left(\mathrm{Re}(z),\mathrm{Im}(z),\mathrm{Re}\left(\frac{e^{\bz Z}-1}{\bz}\right),\mathrm{Im}\left(\frac{e^{\bz Z}-1}{\bz}\right)\right)\,.$$
\end{remark}

To conclude the discussion on elliptic Poisson structures with nonzero elliptic residue, we note that the map $(\omega^\flat)^{-1}\circ\rho^*:T^*M\to \mathcal{A}_{|\{0\}|}$ defines a morphism of Lie algebroids, thus by Lie's second theorem for Lie algebroids (see \cite{MM03}), there should be an associated morphism of Lie groupoids from the $s$-simply connected integration of $\mathcal{A}$ to the elliptic groupoid $\mathcal{H}$ of $(\mathbb{R}^2,\{0\})\,.$ However, we can do a bit better in this case, there already is a map from $\varphi:\mathcal{G}\to\mathcal{H}\,.$ To see this, note that $\mathcal{H}$ has global coordinates given by $(z,w)\in \mathbb{C}\times\mathbb{C}^*\,,$ where $s(z,w)=zw$ and $t(z,w)=z\,.$ Thus, we compute
\begin{proposition}
There is a canonical Lie groupoid morphism $\varphi:\mathcal{G}\to\mathcal{H}$ covering $id_{\mathbb{R}^2}\,,$ given by
$$\varphi(x_1,x_2,a,b)=(x_1+ix_2,(a+bi)(x_1-ix_2)+1)\,.$$
Moreover, this morphism integrates $(\omega^\flat)^{-1}\circ\rho^*$ and pulls back $(t_{\mathcal{H}})^*\omega-(s_{\mathcal{H}})^*\omega$ to $\Omega$ from Corollary \ref{MultSymp1}.
\end{proposition}
\begin{proof}
To show this, we note that it is enough to show that the map is compatible with $s$ and $t\,,$ because then all other identities of a Lie groupoid morphism are satisfied on the pair groupoid of $\mathbb{R}^2\setminus\{0\}$, meaning they immediately extend to the entirety of $\mathcal{G}\,.$ Moreover, we also know that $(\omega^\flat)^{-1}\circ\rho^*$ is just the identity on $\mathbb{R}^2\setminus\{0\}\,,$ so if $\varphi$ is indeed a morphism of Lie groupoids, density also tells us that $\varphi$ integrates $(\omega^\flat)^{-1}\circ\rho^*\,.$

We see
\begin{align*}
s_{\mathcal{H}}\circ\varphi(x_1,x_2,a,b)&=s_{\mathcal{H}}(x_1+ix_2,(a+bi)(x_1-ix_2)+1)\\
&=(a(x_1^2+x_2^2)+x_1)+i(b(x_1^2+x_2^2)+x_2)\\
&=s_{\mathcal{G}}(x_1,x_2,a,b)\\
t_{\mathcal{H}}\circ\varphi(x_1,x_2,a,b)&=x_1+ix_2\\
&=t_{\mathcal{G}}(x_1,x_2,a,b)\,,
\end{align*}
where we used the diffeomorphism $\mathbb{R}^2\cong\mathbb{C}$ sending $(x_1,x_2)$ to $x_1+ix_2\,.$ Finally, we see that $\varphi^*((t_{\mathcal{H}})^*\omega-(s_{\mathcal{H}})^*\omega)=\Omega$ follows from the fact that $\varphi$ is a Lie groupoid morphism and that $\Omega=(t_{\mathcal{G}})^*\omega-(s_{\mathcal{G}})^*\omega\,.$ This completes the proof
\end{proof}

\subsection{Zero elliptic residue}
For the elliptic Poisson structure with no elliptic residue, we will use the local model given by $M=\mathbb{R}^4\,,$ $D=\{0\}\times\mathbb{R}^2$ and $\omega=d\log r\wedge dx_3+d\theta\wedge dx_4\,,$ where $(r,\theta)$ are polar coordinates on the $(x_1,x_2)$-plane. In this case, we see
$$\pi = r\partial_r\wedge\partial_{x_3}+\partial_\theta\wedge\partial_{x_4}\,,$$
hence we get for the cotangent Lie algebroid
$$\mathcal{A}\cong\mathrm{span}(x_1\partial_{x_3}+x_2\partial_{x_4},x_1\partial_{x_4}-x_2\partial_{x_3},r\partial_r,\partial_{\theta})\,.$$
We see that $\mathcal{A}$ has one dense leaf given by $\{(x_1,x_2)\neq 0\}$ and many zero dimensional leaves, namely every point in the plane $\{(x_1,x_2)=0\}\,.$ We note that we can equip $M$ with complex coordinates $z=x_1+ix_2$ and $w=x_3+ix_4\,,$ in which case $\mathcal{A}$ corresponds to the holomorphic Lie algebroid $\mathcal{A}^\hol\cong\mathrm{span}(z\partial_w,z\partial_z)$ as in Remark \ref{Holomorphic}.

As usual, we will proceed by doing a blow-up. This time, it will be blowing up $\{(0,z,0,z):z\in\mathbb{C}\}\subseteq\mathbb{C}^4\,,$ giving a blow-down map $\beta:\widetilde{\mathbb{C}^4}\to\mathbb{C}^4\,.$ Define
$$\mathcal{G}:=\widetilde{\mathbb{C}^4}\setminus \overline{\beta^{-1}[(\{0\}\times\mathbb{C}\times\mathbb{C}\setminus\{0\}\times\mathbb{C})\cup(\mathbb{C}\setminus\{0\}\times\mathbb{C}\times \{0\}\times\mathbb{C})]}\,.$$
We see that $\mathcal{G}$ has global coordinates given by $\{(z,a,b,c)|b\neq 0\}\,,$ such that $\beta(z,a,b,c)=(a,z,ab,ac+z)\,.$
\begin{theorem}
The pair groupoid structure on $(\mathbb{C}^*\times\mathbb{C})^{\times 2}$ extends uniquely to a Lie groupoid structure $\mathcal{G}\rightrightarrows\mathbb{C}^2\,,$ under the diffeomorphism $\beta:\beta^{-1}[(\mathbb{C}^*\times\mathbb{C})^{\times 2}]\to(\mathbb{C}^*\times\mathbb{C})^{\times 2}\,,$ which integrates $\mathcal{A}\,,$ the cotangent Lie groupoid of the elliptic Poisson structure with zero residue.
\end{theorem}
\begin{proof}
Let $U:=\beta^{-1}[(\mathbb{C}^*\times\mathbb{C})^{\times 2}]\subseteq\mathcal{G}\,.$ We have $U=\{(z,a,b,c):a\neq 0,b\neq 0\}\,,$ which is clearly dense in $\mathcal{G}\,,$ so like before, we just need to show that all groupoid structure maps extend from $U$ to $\mathcal{G}\,,$ all groupoid identities and uniqueness of the extension then follow because $U$ is dense in $\mathcal{G}\,.$ We see that on $U\,,$ 
\begin{align*}
s(z,a,b,c)&=\pi_2\circ\beta(z,a,b,c)=(ab,ac+z)\\
t(z,a,b,c)&=\pi_1\circ\beta(z,a,b,c)=(a,z)\,.
\end{align*}
These extend to $\mathcal{G}$ by $s(z,0,b,c)=(0,z)$ and $t(z,0,b,c)=(0,z)\,.$ For $\iota\,,$ we get
$$\iota(z,a,b,c)=(ac+z,ab,b^{-1},-b^{-1}c)\,.$$
This extends to $\mathcal{G}$ by $\iota(z,0,b,c)=(z,0,b^{-1},-b^{-1}c)\,.$ For $u\,,$ we have that if $w\neq 0\,,$
$$u(w,z)=(z,w,1,0)\,,$$
which extends to $\mathbb{C}^2$ by $u(0,z)=(z,0,1,0)\,.$ Lastly, we have on $U$
$$m((z,a,b,c),(ac+z,ab,b',c'))=(z,a,bb',c+b'c)\,.$$
This extends to $\mathcal{G}\times_M\mathcal{G}$ by
$$m((z,0,b,c),(z,0,b',c'))=(z,0,bb',c+b'c)\,.$$
So $(\mathcal{G},s,t,m,\iota,u)$ becomes a Lie groupoid.

To prove that this indeed integrates the correct algebroid, we use Remark \ref{Holomorphic} to compute
$$ds^\hol=\begin{pmatrix}
0&b&a&0\\
1&c&0&a
\end{pmatrix}\,.$$
So we see 
$$\ker(ds^\hol)|_u=\mathrm{span}(a\partial_a-\partial_b,a\partial_z-\partial_c)\,.$$
Applying
$$dt^\hol=\begin{pmatrix}
0&1&0&0\\
1&0&0&0
\end{pmatrix}\,,$$
we see $dt^\hol(\ker(ds^\hol)|_u)=\mathrm{span}(w\partial_w,w\partial_z)\,,$ which is precisely $\mathcal{A}^\hol\,,$ thus $\mathcal{G}$ indeed integrates $\mathcal{A}\,.$
\end{proof}
\begin{remark}
Alternatively, the groupoid above can be obtained as the action groupoid of the $\mathbb{C}^*\ltimes\mathbb{C}$-action on $\mathbb{C}^2$ given by 
$$(w,z)\cdot(z_1,z_2)=(wz_1,z_2+wz)\,.$$
\end{remark}

The following is evident:
\begin{proposition}
$\mathcal{G}$ has the following properties:
\begin{enumerate}
    \item The orbits of $\mathcal{G}$ are $\mathbb{C}^*\times\mathbb{C}$ and all the points of $\{0\}\times \mathbb{C}\,$;
    \item The isotropy of $\mathcal{G}$ over $\mathbb{C}^*\times\mathbb{C}$ is trivial;
    \item The isotropy of $\mathcal{G}$ over any point $p\in\{0\}\times\mathbb{C}$ is $\mathbb{C}^*\ltimes\mathbb{C}\,$, where $\mathbb{C}^*$ acts on $\mathbb{C}$ by complex multiplication.
\end{enumerate}
\end{proposition}
Again, the multiplicative symplectic structure on $\mathcal{G}$ that is associated to this elliptic Poisson structure is obtained by taking $\Omega=t^*\omega-s^*\omega$ on the pair groupoid of $M\setminus D$ and then extending to $\mathcal{G}\,,$ i.e.
\begin{corollary}\label{MultSymp2}
The multiplicative symplectic structure on $\mathcal{G}$ associated to $\omega=d\log w\wedge dz$ is given by
$$\Omega=\frac{c}{b}da\wedge db-da\wedge dc-\frac{1}{b}db\wedge dz+\frac{a}{b}db\wedge dc\,.$$
\end{corollary}

To conclude, we note that in this case there also is a Lie groupoid morphism from $\mathcal{G}$ to the elliptic groupoid $\mathcal{H}$ of $\mathcal{A}_{|D|}\,,$ that integrates $(\omega^\flat)^{-1}\circ\rho^*:T^*M\to \mathcal{A}_{|D|}\,.$ We note that $\mathcal{H}$ has global coordinates $(z_1,z_2,w_1,w_2)\in\mathbb{C}\times\mathbb{C}^*\times\mathbb{C}^2$ with $s(z_1,z_2,w_1,w_2)=(z_1z_2,w_2)$ and $t(z_1,z_2,w_1,w_2)=(z_1,w_1)$ We have the following
\begin{proposition}
There is a canonical Lie groupoid morphism $\varphi:\mathcal{G}\to\mathcal{H}\,,$ covering $\id_{\mathbb{C}^2}\,,$ given by
$$\varphi(z,a,b,c)=(a,b,z,ac+z)\,.$$
Moreover, this Lie groupoid morphism integrates $(\omega^\flat)^{-1}\circ\rho^*$ and pulls back $(t_{\mathcal{H}})^*\omega-(s_{\mathcal{H}})^*\omega$ to $\Omega$ from Corollary \ref{MultSymp2}.
\end{proposition}
\begin{proof}
Like before, we just need to show that $\varphi$ is compatible with $s$ and $t\,,$ then it satisfies all the right Lie groupoid morphism identities on the pair groupoid of $\mathbb{C}^*\times\mathbb{C}\,,$ so by density it will be a Lie groupoid morphism globally. Moreover, we also know that $(\omega^\flat)^{-1}\circ\rho^*$ is just the identity on $\mathbb{C}^*\times\mathbb{C}\,,$ so if $\varphi$ is indeed a morphism of Lie groupoids, density also tells us that $\varphi$ integrates $(\omega^\flat)^{-1}\circ\rho^*\,.$

We see 
\begin{align*}
s_{\mathcal{H}}\circ\varphi(z,a,b,c)&=s_{\mathcal{H}}(a,b,z,ac+z)\\
&=(ab,ac+z)\\
&=s_{\mathcal{G}}(z,a,b,c)\\
t_{\mathcal{H}}\circ\varphi(z,a,b,c)&=(a,z)\\
&=t_{\mathcal{G}}(z,a,b,c)\,.
\end{align*}
That $\varphi^*((t_{\mathcal{H}})^*\omega-(s_{\mathcal{H}})^*\omega)=\Omega$ follows from the fact that $\varphi$ is a Lie groupoid morphism and that $\Omega=(t_{\mathcal{G}})^*\omega-(s_{\mathcal{G}})^*\omega\,.$ This completes the proof
\end{proof}

\printbibliography
\end{document}